\documentclass[11pt]{amsart}
\usepackage{amsmath,amsfonts,amssymb,mathrsfs}
\usepackage{amssymb,mathrsfs,graphicx,enumerate,mathabx}
\usepackage{amssymb,amscd,amsthm,bbm}
\usepackage{mathtools}

\usepackage{algorithm}
\usepackage{algorithmic}

\usepackage{subcaption}
\usepackage{graphicx}
\usepackage[retainorgcmds]{IEEEtrantools}
\usepackage{colortbl}
\usepackage{lipsum}
\usepackage{graphicx}
\usepackage{graphicx}
\usepackage{hyperref}

\title[]{Collective Optimization on Riemannian manifolds with bounded curvature}

\author[H. Huang]{Hui Huang}
\address[Hui Huang]{\newline School of Mathematics, \newline Hunan University, Changsha 410000, China}
\email{huihuang1@hnu.edu.cn}

\author[D. Kim]{Dohyun Kim}
\address[Dohyun Kim]{\newline Department of Mathematics Education and  Institute of Pure and Applied Mathematics, \newline Sungkyunkwan University, Seoul 03063, Republic of Korea}
\email{dohyunkim@skku.edu}

\author[H. Park]{Hansol Park}
\address[Hansol Park]{\newline Department of Mathematics, \newline
National Tsing Hua University,  Hsinchu 30013, Taiwan}
\email{hansolpark@math.nthu.edu.tw}

\thanks{\textbf{Acknowledgment.}
The work of H. Huang was supported by the starting grant from Hunan University. The work of D. Kim  was supported by National Research Foundation of Korea (NRF) grant funded by the Korea government (MSIT) (RS-2024-00454452). The work of H. Park was supported by the National Science and Technology Council (NSTC), Taiwan (Grant No. NSTC 115-2115-M-007-001-MY3).  }

\newtheorem{theorem}{Theorem}[section]
\newtheorem{lemma}{Lemma}[section]

\newtheorem{proposition}{Proposition}[section]
\newtheorem{remark}{Remark}[section]

\newtheorem{definition}{Definition}[section]

\newcommand{\bbr}{\mathbb R}

\newcommand{\bbs}{\mathbb S}

\newcommand{\calC}{\mathcal{C}}

\newcommand{\calM}{\mathcal{M}}
\newcommand{\calP}{\mathcal{P}}

\newcommand{\kp}{\kappa}

\newcommand{\bE}{\mathbb{E}}

\def\d{\mathrm{d}}

\newcommand{\dm}{d} 

\newcommand{\p}{o}
\newcommand{\dist}{d}
\newcommand{\inj}{\mathrm{inj}}
\newcommand{\secc}{\mathcal{K}}

\begin{document}


\begin{abstract}
In this paper, we develop an intrinsic consensus-based optimization framework on Riemannian manifolds with bounded sectional curvature. In contrast to extrinsic approaches based on an ambient Euclidean embedding, our model is formulated directly in terms of the Riemannian structure, using logarithmic and exponential maps induced by the intrinsic geodesic distance. We prove the global well-posedness of the proposed particle system and its associated McKean--Vlasov dynamics. We also establish the global convergence of the mean-field equation toward a global minimizer of the objective function under suitable conditions. Numerical experiments on the sphere, hyperbolic space, and the special orthogonal group demonstrate the effectiveness of the intrinsic CBO dynamics for nonconvex optimization problems on manifolds.
\end{abstract}

\maketitle 

\section{Introduction}
Global optimization in high dimensions is a central task in scientific computing, data science, and machine learning. Beyond classical gradient-based methods, \emph{consensus-based optimization} (CBO) \cite{pinnau2017consensus} has emerged as a robust, derivative-free paradigm inspired by interacting particle systems and collective behavior. In its basic form, CBO evolves a cloud of agents (or particles) through a combination of (i) a drift toward a \emph{consensus point}, computed as a weighted average biased toward low-energy states, and (ii) a state-dependent diffusion that promotes exploration. 
The method has been studied extensively from algorithmic, probabilistic, and mean-field perspectives. In particular, mean-field limits yield nonlinear McKean--Vlasov type equations that enable an SDE-based analysis of stability, concentration, and convergence; see, e.g., \cite{borghi2024kinetic, carrillo2021consensus,fornasier2026consensus,grassi2021particle,huang2025faithful} and references therein. 

\medskip

\noindent\textbf{CBO in Euclidean space.}
Let $\mathcal{E}:\mathbb{R}^d\to\mathbb{R}$ be an objective function. A prototypical task is the global minimization problem:
\begin{equation*}
\text{Find } x^* \in \arg\min_{x\in\mathbb{R}^d} \mathcal{E}(x),
\end{equation*}
where $\mathcal{E}$ may be nonconvex and possibly nonsmooth. For a probability measure $\rho$ on $\mathbb{R}^d$ and a fixed $\alpha>0$, we define the weight function
\[
w_\alpha(x):=\exp(-\alpha\mathcal{E}(x)),
\]
and the associated \emph{consensus point}
\begin{align}\label{eq: consensus}
x_\alpha[\rho]:=\frac{\int_{\mathbb{R}^d} w_\alpha(y)y\,\rho(\mathrm{d}y)}{\int_{\mathbb{R}^d} w_\alpha(y)\,\rho(\mathrm{d}y)}.
\end{align}
This choice of weight function is motivated by the well-known Laplace principle \cite{MR2571413,miller2006applied}; see also \cite[Appendix A.2]{huang2024consensus}. 

At the particle level, one considers $N$ particles $\{x_i(t)\}_{i=1}^N$ in Euclidean space driven by a drift toward $x_\alpha[\rho_t^N]$ with the empirical measure $\rho_t^N :=\frac{1}{N}\sum_{i=1}^N\delta_{x_i(t)}$ and a multiplicative noise proportional to the distance from consensus. They satisfy the following system of stochastic differential equations (SDEs):
\begin{equation}\label{CBO: particle}
\d x_i(t)=-\lambda(x_i(t)-x_\alpha[\rho_t^N])\d t+\sigma \|x_i(t)-x_\alpha[\rho_t^N]\|\d B_t^{i},\quad i=1,\dots,N\,,
\end{equation}
where $\lambda>0$ controls consensus formation, $\sigma>0$ controls exploration, and $\{B^i\}_{i=1}^N$ are independent Brownian motions.

In the mean-field regime, the empirical measure $\rho_t^N$ converges (formally, and in many settings rigorously) to a law $\rho_t$ solving a nonlinear Fokker--Planck equation of the form:
\begin{equation}\label{eq:intro-euclidean}
\partial_t\rho_t
=
\lambda\,\nabla\cdot\!\big((x-x_\alpha[\rho_t])\rho_t\big)
+\frac{\sigma^2}{2}\Delta\!\big(\|x-x_\alpha[\rho_t]\|^2\rho_t\big)\,.
\end{equation}
Moreover, the corresponding McKean--Vlasov SDE is given by:
\begin{equation}\label{CBO: Mckean}
\d \bar{x}(t)=-\lambda(\bar{x}(t)-x_\alpha[\rho_t])\d t+\sigma \|\bar{x}(t)-x_\alpha[\rho_t]\|\d B_t\,,\quad \rho_t=\mbox{Law}(\bar{x}(t))\,.
\end{equation}
Rigorous derivations of the mean-field limit for the particle system \eqref{CBO: particle}, leading to the PDE  \eqref{eq:intro-euclidean} and the McKean--Vlasov SDE \eqref{CBO: Mckean}, are provided in \cite{bolley2011stochastic,gerber2023mean,huang2025uniform,huang2022mean}. The large-$\alpha$ regime concentrates the weight $w_\alpha$ near minimizers of $\mathcal{E}$, and the coupling between drift and diffusion promotes the collapse of $\rho_t$ around (near-)global minimizers under suitable assumptions \cite{byeon2025consensus,carrillo2018analytical,fornasier2024consensus}.

The CBO algorithm offers several key advantages, notably its derivative-free nature and   strong amenability to rigorous mathematical analysis. Consequently, the method has been extensively developed and adapted to a wide range of complex settings, including multiple-minimizer problems \cite{bungert2025polarized}, stochastic optimization problems \cite{bellavia2025discrete}, multi-level optimization \cite{herty2025multiscale}, multi-objective optimization \cite{borghi2023adaptive}, constrained optimization \cite{beddrich2026constrained}, optimization in Sobolev spaces \cite{khatab2026consensus}, CBO via jump diffusions \cite{kalise2023consensus},  uniform-in-time propagation of chaos \cite{gerber2025uniform, ha2026uniform}, mirrorCBO \cite{bungert2025mirrorcbo}, discrete CBO \cite{ha2024time,ha2020convergence,ko2022convergence}, one-dimensional CBO \cite{choi2022one}, and CBO with memory effects \cite{riedl2024leveraging}.  Rather than attempting to provide an exhaustive account of this rapidly expanding field, we refer the interested reader to the survey \cite{totzeck2021trends} and the more recent comprehensive review \cite{byeon2025consensus,fornasier2026consensus} for a broader perspective.

More closely related to our current work, CBO methods have been applied to optimization on manifolds \cite{fornasier2020consensus,fornasier2021JMLR,fornasier2022anisotropic,ha2022stochastic}. In particular, the authors of \cite{fornasier2020consensus} introduced a CBO particle system for the global optimization of nonconvex functions over a connected, smooth, compact hypersurface $\Gamma$ embedded in $\mathbb{R}^d$. The rigorous proof of global convergence toward a global minimizer was subsequently established in \cite{fornasier2021JMLR} for the specific case of the sphere ($\Gamma=\mathbb{S}^{d-1}$). This framework was later extended to incorporate anisotropic diffusion, enabling the model to handle high-dimensional optimization problems constrained to the unit hypersphere \cite{fornasier2022anisotropic}. Furthermore, \cite{ha2022stochastic} considered a CBO method for minimizing a nonconvex function over the Stiefel manifold, which is the set of all orthonormal frames in Euclidean space, where both the mean-field limit and global convergence were established. However, a structural limitation of all the aforementioned models is that  the manifold is treated \emph{extrinsically} by embedding it in the Euclidean space $\mathbb{R}^d$. Consequently, the aforementioned extrinsic models  rely on ambient projections rather than fully using the intrinsic geometric information of the space where the optimization problem is constrained.
\medskip

\noindent\textbf{Why manifolds?}
In many applications, the decision variable is naturally constrained to a curved space: orientations ($\mathbb{S}^{d-1}$, $SO(d)$), positive definite matrices (SPD manifolds), shape spaces, low-rank models, and statistical manifolds. In such cases, enforcing constraints via ambient projections can distort the underlying geometry and degrade algorithmic performance. This motivates the development of intrinsic CBO dynamics posed directly on a Riemannian manifold $(M,g)$, utilizing only geometric primitives such as geodesic distance, exponential/logarithmic maps, and the Laplace--Beltrami operator. 

In several previous works, CBO dynamics on manifolds have been formulated through an extrinsic viewpoint 
\cite{carrillo2405interacting,fornasier2020consensus,fornasier2024consensus,ha2022stochastic}. In these approaches, the manifold is embedded into an ambient Euclidean space, and the CBO dynamics are constructed using the metric structure induced from the ambient space. However, such an extrinsic formulation has the limitation that it treats the given manifold not as an abstract Riemannian manifold, but rather as an embedded submanifold of a particular Euclidean space. Consequently, the resulting dynamics are not only determined by the intrinsic geometry of the manifold, but are also affected by the chosen embedding and the ambient metric. In particular, the ambient Euclidean distance is generally different from the geodesic distance on the manifold, and the consensus direction or diffusion term defined in the ambient space need not coincide with those naturally induced by the Riemannian metric.

This issue becomes more significant when one considers general Riemannian manifolds. For manifolds such as spheres or Stiefel manifolds, which admit natural Euclidean embeddings, an extrinsic approach can be relatively convenient. However, a general manifold does not come with a canonical ambient space. Even if an embedding exists, for instance by the Nash embedding theorem, the choice of such an embedding is not unique, and the extrinsic dynamics induced by it may also be embedding-dependent. Therefore, to study CBO on general Riemannian manifolds, it is necessary to develop an intrinsic formulation that relies only on the Riemannian structure of the manifold itself, without referring to an auxiliary ambient space.

 Another advantage of an intrinsic formulation is its computational aspect. When an $n$-dimensional manifold is represented as an embedded submanifold of a Euclidean space $\mathbb{R}^d$, the extrinsic dynamics are often implemented in the ambient dimension $d$, which may be larger than the intrinsic dimension $n$. In some cases, this increase in dimension is small; for example, the sphere $\mathbb{S}^n$ is naturally embedded in $\mathbb{R}^{n+1}$. In other cases, the ambient dimension can be much larger than the intrinsic dimension. For example, $SO(n)$ has intrinsic dimension $n(n-1)/2$, whereas its standard matrix embedding lies in $\mathbb{R}^{n^2}$. Therefore, depending on the manifold, the dimension of the Euclidean space required for an embedding can become quite large. An intrinsic formulation avoids this artificial enlargement of the state space and can therefore offer computational advantages in the implementation of CBO dynamics on manifolds.

For these reasons, in this paper we focus on developing a CBO framework based on the intrinsic geometry of the manifold. The intrinsic formulation of dynamical systems on manifolds has already been studied in the context of collective behavior and aggregation models. For instance, Fetecau and Zhang introduced self-organization models on Riemannian manifolds in which the interactions are defined intrinsically through the Riemannian distance and the associated geometric structure \cite{FeZh2019}. Subsequently, the well-posedness and long-time behavior of intrinsic dynamics on $SO(3)$ and  the sphere were studied in \cite{FeHaPa2021,FePaPa2020}, respectively. More recently, intrinsic dynamics were investigated on Riemannian manifolds with bounded curvature, rather than on a specific model manifold \cite{FePa2023b}. See also \cite{yang2026riemannian} for a Riemannian gradient method on the Cartan--Hadamard manifold for $L_2$ Wasserstein least-squares problems on the space of Gaussian measures equipped with the affine-invariant geometry. We refer the reader to \cite{FeHaPa2021,FePa2023b,FePaPa2020,FeZh2019} for further developments on intrinsic collective dynamics on Riemannian manifolds. However, these works concern aggregation and self-organization dynamics, rather than consensus-based optimization. Motivated by these developments, we formulate consensus-based optimization dynamics intrinsically on a general Riemannian manifold with bounded curvature in this paper.

\medskip

\noindent\textbf{Geometric difficulties.}
Extending \eqref{eq:intro-euclidean} to a general manifold is not a mere notational change. Two structural obstacles are immediate:

\begin{enumerate}
\item \emph{No subtraction.} The Euclidean drift $(x-x_\alpha[\rho])$ and the average $\int (y-x)\rho(\mathrm{d}y)$ are not well-defined intrinsically. A natural replacement is to use logarithmic vectors $\log_x(y)\in T_xM$, such that $\|\log_x(y)\|=d(x,y)$, and to transport vectors between tangent spaces using parallel transport.
\item \emph{Cut-locus and curvature effects.} Even on complete manifolds, $\log_x(y)$ may fail to be globally defined due to the cut-locus. Furthermore, curvature affects both the Lipschitz properties of the logarithmic map and the strength of the diffusion. On negatively curved manifolds in particular, diffusion can inflate moments much more strongly than in $\mathbb{R}^d$, complicating the \emph{a priori} estimates typically relied upon in Euclidean theory.
\end{enumerate}

These issues suggest that a viable intrinsic CBO model must (a) ensure the logarithmic map is evaluated only in a domain where it is single-valued and regular, and (b) control the diffusion in regions where curvature-driven dispersion would prevent moment closure.

\medskip

\noindent\textbf{Our approach and main model.}
In this work, we propose and analyze CBO dynamics on a class of manifolds with bounded sectional curvature. We place particular emphasis on Cartan--Hadamard manifolds (complete, simply connected and non-positively curved spaces) where geodesics are unique and global convexity properties are available. Throughout this paper, we make the following assumptions:\\

\noindent\textbf{(M)} \textit{$M$ is a complete, connected, smooth Riemannian manifold without boundary of dimension $\dm$, with positive injectivity radius $\inj(M)>0$. We denote by $\dist(\cdot,\cdot)$ the intrinsic distance on $M$. Moreover, every sectional curvature $\mathcal{K}(x;\sigma)$ of $M$, computed at a point $x\in M$ and a $2$-dimensional subspace $\sigma\subset T_xM$, satisfies $-\kappa_- \leq \mathcal{K}(x;\sigma) \leq \kappa_+$ for positive constants $\kappa_\pm$.}\\

\noindent\textbf{(R)} \textit{Given a manifold $M$ satisfying \textbf{(M)}, we choose positive constants $R$ and $\delta \in (0,R)$ such that $0<R<R+\delta<\min\left\{\frac{\inj(M)}{2},\,\frac{\pi}{2\sqrt{\kappa_+}}\right\}$.}\\

\noindent\textbf{(E)} \textit{$\mathcal{E}$ is globally Lipschitz on $M$ and $\inf_{x\in M} \mathcal E(x) =:\underline{\mathcal E}$.}\\

This assumption on $\mathcal{E}$ only guarantees the well-posedness of the dynamics. To obtain the long-time behavior, we  impose the additional assumptions (A1)--(A3) in Section \ref{sec:4}.\\

Consequently, if we consider a geodesic ball $B_R(\p):=\{y\in M:\dist(\p,y)<R\}$ centered at a pole $\p$ with radius $R<\min\left(\frac{\mathrm{inj}(M)}{2}, \frac{\pi}{2\sqrt{\kappa_+}}\right)$, then for any two points $x,y\in B_R(\p)$, the minimizing geodesic joining them is unique (see Section \ref{sec:2.1}).

To address the lack of vector subtraction, we construct an intrinsic analog of the consensus difference $x_\alpha[\rho]-x$. In Euclidean space, this difference is the weighted average of $(y-x)$, that is, 
\[
x_\alpha[\rho]-x
=
\frac{1}{\int_{\bbr^d} w_\alpha(y)\, d\rho(y)}
\int_{\bbr^d} w_\alpha(y)(y-x)\, d\rho(y).
\]
 On a manifold, we replace this with the weighted average of $\log_x y$, i.e.
 \[
\frac{1}{\int_M w_\alpha(y)\, d\rho(y)}
\int_M w_\alpha(y)\log_x y\, d\rho(y).
\]
To avoid cut-locus singularities and ensure $\log_x y$ is globally well-defined, we restrict the integration domain to a geodesic ball $B_{R+\delta}(\p)$ with $0<R<R+\delta<\min\left(\frac{\mathrm{inj}(M)}{2}, \frac{\pi}{2\sqrt{\kappa_+}}\right)$. By the triangle inequality, any two points $x,y\in B_{R+\delta}(\p)$ satisfy $\dist(x,y) < 2(R+\delta) <\min\left(\mathrm{inj}(M), \frac{\pi}{\sqrt{\kappa_+}}\right)$, ensuring the required regularity. Hence, $\log_x y$ is well-defined for all $x,y\in B_{R+\delta}(\p)$.

To maintain continuous dependence on the measure $\rho$, we introduce a smooth cutoff function $h_{R,R+\delta}$ and define the intrinsic consensus direction $u_\alpha(\rho; x) \in T_xM$ at base point $x$:
\[
u_\alpha(\rho; x):=
\frac{h_{R,R+\delta}(\dist(x,\p))}{\int_M w_\alpha(y)\, \d\rho(y)}
\int_{B_{R+\delta}(\p)} h_{R,R+\delta}(\dist(y,\p))\, w_\alpha(y)\log_x y\, \d\rho(y)\,.
\]
 Namely, the vector $u_\alpha(\rho;x)\in T_xM$ serves as the intrinsic counterpart of $x_\alpha[\rho]-x$.
Unlike the Euclidean case as in \eqref{eq: consensus}, the corresponding consensus point $x_\alpha(\rho;x):=\exp_x\bigl(u_\alpha(\rho;x)\bigr)$ depends explicitly on the base point $x$. Here the smooth non-increasing function $h:\mathbb{R}\to[0,1]$ satisfies
\[
h(\theta)=1 \quad \text{for all } \theta\le 0,
\qquad
h(\theta)=0 \quad \text{for all } \theta\ge 1.
\]
For any two real numbers \(a<b\), we define the rescaled cutoff function $h_{a,b}$ by
\[
h_{a,b}(\theta) :=h\!\left(\frac{\theta-a}{b-a}\right).
\]
Then \(h_{a,b}\) is also smooth and non-increasing, and satisfies
\[
h_{a,b}(\theta)=1 \quad \text{for all } \theta\le a,
\qquad
h_{a,b}(\theta)=0 \quad \text{for all } \theta\ge b.
\]
As an example, we can construct $h$ as follows:
\[
h(\theta)=\frac{s(1-\theta)}{s(\theta)+s(1-\theta)}, \qquad s(\theta) :=
\begin{cases}
\exp\!\left(-\frac{1}{\theta}\right), & \forall\,\theta>0,\\
0, & \forall\,\theta\le0.
\end{cases}
\]

The diffusion term also requires a careful modification. The required modification is not merely a formal replacement of the Euclidean Laplacian by the Laplace--Beltrami operator $\Delta_M$ on $M$. Because the volume of geodesic balls on Cartan--Hadamard manifolds grows much faster, standard diffusion $\Delta_M\bigl(\|u_\alpha^t(x)\|^2\rho_t(x)\bigr)$ is substantially stronger than in Euclidean space, preventing the derivation of Gr\"onwall-type inequalities for higher-order moments. To counteract this dispersive effect, we apply a smooth cutoff $h_{R-\delta,R}(\dist(x,\p))$ to the diffusion term outside the compact active region. To avoid introducing an additional parameter, we use the same $\delta$ as in the cutoff for the drift term. Accordingly, we further assume that $\delta\in (0,R)$.

Combining these localized geometric adaptations, we are now ready to  formulate the $N$-particle CBO system on $M$ to solve the problem
\begin{equation*}
\text{Find } x^* \in \arg\min_{x\in M} \mathcal{E}(x)\,.
\end{equation*} 
Let $x_i(t)$ be the position of the $i$-th particle at time $t>0$. Then, our fully intrinsic CBO system is driven by the following system of SDEs:
\begin{equation}\label{eq: particle}
\begin{aligned}
\d x_i(t) &=\lambda u_\alpha(\rho^{N, X(t)};x_i(t))\d t+\sigma h_{R-\delta, R}(\dist(x_i(t), \p))\|u_\alpha(\rho^{N, X(t)};x_i(t))\| \d B_i^M(t)\\
&=\lambda u_\alpha(\rho^{N, X(t)};x_i(t))\d t+\sigma h_{R-\delta, R}(\dist(x_i(t), \p))\|u_\alpha(\rho^{N, X(t)};x_i(t))\| \sum_{k=1}^\dm E_k(x_i(t))\d B_i^k(t),
\end{aligned}
\end{equation}
where $\rho^{N, X(t)}=\frac{1}{N}\sum_{i=1}^N\delta_{x_i(t)}$ is the empirical measure, $\{E_k(x)\}_{k=1}^\dm$ is a local orthonormal frame on $T_xM$, and $\{B_i^k(t)\}_{k=1}^\dm$ are i.i.d. one-dimensional Brownian motions. 

As $N \to \infty$, \eqref{eq: particle} formally satisfies  the corresponding McKean--Vlasov process on $M$ for the law $\rho_t$ of the limiting process $\bar{x}(t)$:
\begin{equation}\label{eq: Mckean}
\d \bar{x}(t)=\lambda u_\alpha(\rho_t;\bar{x}(t))\d t+\sigma h_{R-\delta, R}(\dist(\bar{x}(t), \p))\|u_\alpha(\rho_t;\bar{x}(t))\| \d B^M(t).
\end{equation}
Here, $B^M(t)$ denotes a Riemannian Brownian motion on $M$. In addition,  the backward generator of the McKean--Vlasov process acting on $f\in C^\infty(M)$ is
\begin{equation} \label{generator}
\mathcal{L}_tf(x):=\langle\lambda u_\alpha(\rho_t;x),\nabla_Mf(x)\rangle+\frac{\sigma^2}{2}h_{R-\delta, R}(\dist(x, \p))^2\|u_\alpha(\rho_t;x)\|^2\Delta_Mf(x)
\end{equation}
and the law $\rho_t$ governs the collective behavior via the intrinsic nonlinear Fokker--Planck equation on $M$:
\begin{align}\label{Main-eq}
\begin{cases}
&\displaystyle\partial_t\rho_t(x)=-\lambda\mathrm{div}_M\left(u_\alpha^t(x)\rho_t(x)\right)+\frac{\sigma^2}{2}\Delta_M\left(h_{R-\delta, R}\left(\dist(x, \p)\right)^2\|u_\alpha^t(x)\|^2\rho_t(x)\right),\vspace{0.3cm}\\
&\displaystyle  u_\alpha^t(x)=\frac{h_{R, R+\delta}(\dist(x, \p))}{\int_{M}w_\alpha(y)\d \rho_t(y)}\int_{B_{R+\delta}(\p)}h_{R, R+\delta}(\dist(y, \p))w_\alpha(y)\log_xy~\d \rho_t(y),
\end{cases}
\end{align}
where $\mathrm{div}_M$ and $\Delta_M$ are the divergence and Laplace--Beltrami operators on $M$, respectively, and we simply write $u_\alpha^t(x) := u_\alpha(\rho_t;x)$.

Compared to the Euclidean model, an essential difference is that the \textit{consensus point} depends on the base point through $\log_x(\cdot)$: the map $x\mapsto u_\alpha(\rho;x)$ is a vector field rather than a constant vector. Moreover, we \emph{localize} both the definition of $u_\alpha$ and the diffusion term via the cutoff function $h$ in order to (i) avoid cut-locus singularities and (ii) recover tractable moment estimates in negatively curved settings.
\\

\medskip

\noindent\textbf{Contributions.}
The main contributions of this paper are summarized as follows:
\begin{itemize}
\item We propose an intrinsic CBO particle system \eqref{eq: particle} and mean-field model \eqref{eq: Mckean} or Fokker--Planck equation \eqref{Main-eq} on manifolds with bounded sectional curvature, based on logarithmic averaging in tangent spaces and a geometric cutoff mechanism.
\item We obtain  well-posedness for \eqref{eq: particle}  and  \eqref{eq: Mckean}  under natural regularity assumptions on the objective function $\mathcal{E}$ and explicit curvature-dependent estimates for the logarithmic map (see Theorem \ref{thm:T3.1} and Theorem \ref{thm:T3.2}). 
\item We rigorously establish the global convergence of the proposed CBO dynamics \eqref{Main-eq}  toward a global minimizer of the objective function in the mean-field limit (see Theorem \ref{thm:convergence}). 
\end{itemize}
These results provide a rigorous PDE or SDE framework for analyzing consensus-based optimization beyond Euclidean space while preserving the intrinsic geometry.

\medskip

We also complement our theoretical analysis with numerical experiments that demonstrate the practical behavior of the proposed intrinsic CBO dynamics \eqref{eq: particle}. The particle system is discretized by a Riemannian  Euler--Maruyama scheme where the intrinsic drift is computed through the logarithmic maps and the particles are updated by the exponential map. We test the algorithm on several representative manifolds, including the sphere $\bbs^2$, the hyperbolic space $\mathbb{H}^2$ and the special orthogonal group $SO(3)$ using highly nonconvex Ackley-type objective functions adapted to each geometric setting. These numerical experiments illustrate that the intrinsic CBO dynamics can successfully guide initially distributed particles toward the global minimizer even in the presence of multiple local minima. Moreover, the observed decay rate of the empirical variance is consistent with the exponential convergence predicted by our analysis.

This paper is organized as follows. In Section \ref{sec:2}, we review the geometric preliminaries for the intrinsic formulation such as injectivity radius, parallel transport and comparison estimates. Section \ref{sec:3} studies the well-posedness of the particle system and McKean--Vlasov process for the CBO dynamics on a Riemannian manifold satisfying \textbf{(M)}. In Section \ref{sec:4}, we establish the global convergence of the mean-field equation toward a global minimizer. Section \ref{sec:5} presents numerical simulations that demonstrate the behavior of the intrinsic CBO algorithm on representative manifolds such as the sphere, hyperbolic space, and $SO(3)$. In Appendix \ref{sec:app.A}, the proofs of several lemmas are provided.


\section{Preliminaries}\label{sec:2}
\setcounter{equation}{0}
In this section, we present the geometric concepts and preparatory lemmas that will be used throughout the paper. First, as assumed in Section 1, we begin by assuming that the given manifold $M$ satisfies condition \textbf{(M)}.

\subsection{Injectivity radius and parallel transport}\label{sec:2.1}
The exponential map is defined as  
\[
\exp_x:T_xM\to M,
\]
and $\exp_x v$ is defined as $\gamma(1)$, where $\gamma$ is the geodesic satisfying $\gamma(0)=x$ and $\dot{\gamma}(0)=v$.

Although $\exp_x$ is well defined, its inverse map $\log_x$ may not be well-defined on all of $M$, since $\exp_x$ is not necessarily injective. To describe the region where $\exp_x$ is injective, we define the injectivity radius $\mathrm{inj}(M)$ as follows: for each $x\in M$, let $\mathrm{inj}_x(M)$ be the supremum of all $r>0$ such that $\exp_x$ is a diffeomorphism from the open ball $B_r(\mathbf{0})\subset T_xM$ onto its image where $\mathbf{0}$ is the ($d$-dimensional) zero vector. Then, we define
\[
\mathrm{inj}(M):=\inf_{x\in M}\operatorname{inj}_x(M).
\]
In particular, if $r<\operatorname{inj}(M)$, then for every $x\in M$, the map $\exp_x$ is injective on $B_r(\mathbf{0})\subset T_xM$, and hence $\log_x$ is well defined on the corresponding geodesic ball $B_r(x)$. 

If we assume 
\begin{align*}\label{condR}
R<\min\left(\frac{\mathrm{inj}(M)}{2}, \frac{\pi}{2\sqrt{\kappa_+}}\right),
\end{align*}
then any two points $x,y\in B_R(\p)$ can be connected by a unique minimizing geodesic lying entirely in $B_R(\p)$ (see  \cite[Theorem 6.4.8]{Petersen2016}).\\

In Euclidean space, vectors at different base points can be compared after translating one of them. On a curved manifold, this should be replaced by parallel transport along a curve connecting two points. More precisely, for $x,y\in M$, we denote  
\[
P_{xy}:T_yM\to T_xM
\]
by the parallel transport along a minimizing geodesic from $y$ to $x$. Whenever the minimizing geodesic between $x$ and $y$ is unique, $P_{xy}$ is well-defined.

\subsection{Comparison theorems}

In this subsection, we present several comparison results on Riemannian manifolds that are related to curvature bounds. From the Laplacian comparison theorem \cite[Lemma 7.1.9]{Petersen2016} and assumption \textbf{(M)}, we obtain the following lemma.
 
\begin{lemma} \label{lem:laplace}
Let $M$ satisfy \textbf{(M)}. Then, for each fixed $y\in M$,
\[
\Delta_x \dist(x,y)
\leq
(\dm-1)\sqrt{\kappa_-}\coth\bigl(\sqrt{\kappa_-}\dist(x,y)\bigr)
\]
for all $x\in M\setminus\bigl(\{y\}\cup \mathrm{Cut}(y)\bigr)$.
\end{lemma}

Indeed, \cite[Lemma 7.1.9]{Petersen2016} is stated under a lower bound on the Ricci curvature. However, under assumption \textbf{(M)}, the sectional curvature bound $-\kappa_- \leq \mathcal{K}(x;\sigma)$ implies the Ricci curvature bound $\mathrm{Ric}(M)\geq-(\dm-1)\kappa_-$. Hence the assumptions of \cite[Lemma 7.1.9]{Petersen2016} are satisfied.\\

Now, we introduce the Rauch comparison theorem.

\begin{lemma}[Rauch comparison theorem \cite{cheeger1975comparison}]\label{RCT}
Let $M$ and $\tilde{M}$ be Riemannian manifolds with $\mathrm{dim} (\tilde{M})\geq \mathrm{dim}(M)$, and suppose that for all $p\in M$, $\tilde{p}\in \tilde{M}$, and $\sigma\subset T_pM$, $\tilde{\sigma}\subset T_{\tilde{p}}\tilde{M}$, the sectional curvatures $\secc$ and $\tilde{\secc}$ of $M$ and $\tilde{M}$, respectively, satisfy
\[
\tilde{\secc} (\tilde{p};\tilde{\sigma})\geq \secc(p;\sigma).
\]
Let $p\in M$, $\tilde{p}\in\tilde{M}$ and fix a linear isometry $i:T_pM\to T_{\tilde{p}}\tilde{M}$. Let $r>0$ be such that the restriction ${\exp_p}_{|B_r(0)}$ is a diffeomorphism and ${\exp_{\tilde{p}}}_{|B_r(0)}$ is non-singular. Let $c:[0, a]\to\exp_p(B_r(0))\subset M$ be any differentiable curve and define  $\tilde{c}:[0, a]\to\exp_{\tilde{p}}(B_r(0))\subset\tilde{M}$ by
\[
\tilde{c}(s)=\exp_{\tilde{p}}\circ i\circ\exp^{-1}_p(c(s)),\qquad s\in[0, a].
\]  
Then the length of $c$ is greater than or equal to the length of $\tilde{c}$.
\end{lemma}

From the Rauch comparison theorem, we can compare the difference between two logarithmic maps at the common base $x\in M$ as follows.

\begin{lemma}\label{lem:logdiff}
Let $M$, $R$, and $\delta$ satisfy \textbf{(M)} and \textbf{(R)}. Then,
\[
\|\log_x y_1-\log_x y_2\|
\le
\frac{2\sqrt{\kappa_+}R}{\sin(2\sqrt{\kappa_+}R)}\,\dist(y_1,y_2)
\]
for all $x, y_1, y_2\in B_R(\p)$.
\end{lemma}

\begin{proof}
The proof is provided in Appendix \ref{proof:lem:logdiff}.
\end{proof}

We next recall a Hessian estimate for the squared Riemannian distance, which will be used later to compare two stochastic trajectories on the manifold.

\begin{lemma}\label{lem:hessian}
Let $z_1, z_2\in M$ with $\dist(z_1, z_2)<\inj(M)$, $v_1\in T_{z_1}M$, $v_2\in T_{z_2}M$. Furthermore, the lower bound of the sectional curvature of $M$ satisfies $\mathcal K\geq-\kappa$ for some $\kappa\geq0$. Then the Hessian of the distance square function $\Psi(z_1, z_2)=:\frac{1}{2}\dist(z_1, z_2)^2$ can be expressed as
\[
\mathrm{Hess}\Psi(z_1, z_2)[(v_1, v_2), (v_1, v_2)]\leq\|v_1-P_{z_1z_2}v_2\|^2+\frac{\kappa \dist(z_1, z_2)^2}{2}\left(\|v_1\|^2+\|v_2\|^2\right).
\]
\end{lemma}

\begin{proof}
We defer the proof to Appendix \ref{proof:lem:hessian}.
\end{proof}

If we apply the above lemma with \textbf{(M)}, then it can be written as
\begin{equation}\label{es: Hess}
    \mathrm{Hess}\Psi(z_1, z_2)[(v_1, v_2), (v_1, v_2)]\leq\|v_1-P_{z_1z_2}v_2\|^2+
\frac{\kappa_-\dist(z_1, z_2)^2}{2}\left(\|v_1\|^2+\|v_2\|^2\right).
\end{equation}

\begin{remark} \label{R2.1}
For fixed $w\in M$ define
\[
\tilde{\Psi}_w(z):=\Psi(w, z).
\]
By letting $v_2=0$, we can find the Hessian inequality of $\tilde{\Psi}$:
\[
\mathrm{Hess}\tilde{\Psi}_w(z)[v, v]\leq \|v\|^2\left(1+\frac{\kappa_-\dist(z, w)^2}{2}\right),\quad\forall~\dist(z, w)<\inj(M),\quad v\in T_zM.
\]
It implies that
\[
\|\mathrm{Hess}\tilde{\Psi}_w(z)\|_{\mathrm{op}}\leq1+\frac{\kappa_-\dist(z, w)^2}{2}.
\]
\end{remark}

As a consequence of the Hessian estimate above, we obtain the following stability estimate for the logarithmic map with respect to the base point.

\begin{lemma}\label{lem:transport-diff}
Let $M$, $R$, and $\delta$ satisfy \textbf{(M)} and \textbf{(R)}. Then, we have
\[
\|\log_{x}y-P_{xz}\log_{z}
y\|\leq C_1\dist(x, z),\quad\forall~x, y, z\in B_{R+\delta}(\p)
\]
where
\[
C_1:=1+2\kappa_-(R+\delta)^2.
\]
\end{lemma}

\begin{proof}

The details of the proof are given in Appendix \ref{proof:lem:transport-diff}.

\end{proof}

We finalize this section by introducing the definition of the weak solution. To this end, we denote $\mathcal P(M)$ as the space of Borel probability measures on $M$ and $\mathcal P_p(M)$ for $p>0$ as 
\[
\mathcal P_p(M):= \left \{ \mu \in \mathcal P(M): \int_M d(x,o)^p d\mu(x) <\infty       \right\}.
\]

\begin{definition}[Weak solution] \label{def:weaksol}
Let $\rho_0\in\mathcal{P}(M)$ and $T>0$. We say $\rho\in C([0, T],\mathcal{P}(M))$ satisfies \eqref{Main-eq} with the initial data $\rho_0$ in the weak sense in the time interval $[0, T]$ if the following equality holds for all $\phi\in {C}_c^\infty(M)$ and all $t\in (0, T)$:
\begin{align*}
\frac{\d}{\d t}\int_M \phi(x)\d\rho_t(x)&=\lambda\int_{M} \nabla_M\phi(x)\cdot u_\alpha^t(x)\d\rho_t(x) \\
&\quad +\frac{\sigma^2}{2}\int_M h_{R-\delta,R}(\dist(x, \p))^2\|u_\alpha^t(x)\|^2\Delta_M\phi(x)\d\rho_t(x),
\end{align*}
and \[
\lim_{t\to 0+}\int_M \psi(x)\d\rho_t(x)=\int_M \psi(x)\d\rho_0(x)
\quad \text{for all } \psi\in C_b(M).
\]
\end{definition}

\section{Well-posedness} \label{sec:3}
\setcounter{equation}{0}

In this section, we establish the global existence and uniqueness of solutions  to \eqref{eq: particle} and \eqref{eq: Mckean}. From now on, we assume that $\mathcal{E}$, $M$, $R$, and $\delta$ satisfy \textbf{(E)}, \textbf{(M)}, and \textbf{(R)}. 

\subsection{Elementary estimates}
Since the velocity vector field plays a crucial role in global well-posedness, we provide several elementary estimates on the velocity vector field:
\[
u_\alpha(\rho;x)=\frac{h_{R,R+\delta}(\dist(x,\p))}{\int_M w_\alpha(y)\,d\rho(y)}
\int_{B_{R+\delta}(\p)} h_{R,R+\delta}(\dist(y,\p))\,w_\alpha(y)\log_x y\,d\rho(y),
\quad x\in M.
\]
For convenience, we   introduce
\[
\tilde{u}_\alpha(\rho;x):=\frac{1}{\int_M w_\alpha(y)\,d\rho(y)}
\int_{B_{R+\delta}(\p)} h_{R,R+\delta}(\dist(y,\p))\,w_\alpha(y)\log_x y\,d\rho(y),
\quad x\in B_{R+\delta}(\p).
\]
Note that $\tilde{u}_\alpha(\rho;x)$ is defined only for $x\in B_{R+\delta}(\p)$. Indeed, if $x\notin B_{R+\delta}(\p)$, then the condition $\dist(x,y)<\inj(M)$ for all $y\in B_{R+\delta}(\p)$ is no longer guaranteed, so the logarithmic map $\log_x y$ may fail to be well-defined. \\

First, we show that $\tilde u_\alpha(\rho;x)$ is uniformly bounded with respect to $\rho$ and $x$.
\begin{lemma}\label{L4.1}
For any $\rho\in\mathcal{P}(M)$ and $x\in B_{R+\delta}(\p)$, we have
\[
\|\tilde{u}_\alpha(\rho;x)\|\leq 2(R+\delta).
\]
\end{lemma}

\begin{proof}
We provide the proof in Appendix  \ref{proof:L4.1}.
\end{proof}

For a fixed probability measure, we next prove that the intrinsic consensus vector field is Lipschitz continuous with respect to the spatial variable, after identifying different tangent spaces by parallel transport.

\begin{lemma}\label{L4.2}
For any $\rho\in\mathcal{P}(M)$ and $x_1, x_2\in B_{R+\delta}(\p)$, we have
\[
\|u_\alpha(\rho;x_1)-P_{x_1x_2}u_\alpha(\rho;x_2)\|\leq C_2\dist(x_1, x_2),
\]
where
\[
C_2:=C_1+2\mathrm{Lip}(h_{R,R+\delta})(R+\delta).
\]
\end{lemma}

\begin{proof}
The proof can be found in Appendix \ref{proof:L4.2}.

\end{proof}

 After controlling the dependence on the base point, we now estimate the dependence of the intrinsic consensus vector field on the probability measure in terms of the Wasserstein 1-distance.

\begin{lemma}\label{L4.3}
For any $\rho_1, \rho_2\in \mathcal{P}(M)$, and $x\in M$, we have
\[
\|u_\alpha(\rho_1;x)-u_\alpha(\rho_2;x)\|\leq C_3W_1(\rho_1, \rho_2),
\]
where
\[
C_3:=\frac{1}{\inf w_\alpha}\bigg(\sup w_\alpha\big(2(R+\delta)\mathrm{Lip}(h_{R,R+\delta})+1\big)+4(R+\delta)\mathrm{Lip}(w_\alpha)\bigg)
\]
\end{lemma}

\begin{proof}
The proof is postponed to Appendix \ref{proof:L4.3}.

\end{proof}

For any two ensembles $X:=\{x_i\}_{i=1}^N$ and $Y:=\{y_i\}_{i=1}^N$, we define the maximal distance between them:
\[
D(X,Y):=\max_{1\leq i\leq N}\dist(x_i, y_i).
\]

 With this notation, we can compare the consensus vector fields generated by two different particle configurations. The following lemma follows from the Lipschitz estimates with respect to both the base point and the empirical measure.

\begin{lemma}\label{lem: Lip-particle}
Let $X=\{x_i\}_{i=1}^N\subset M$ and  $Y=\{y_i\}_{i=1}^N\subset M$ and assume that $D(X, Y) \leq 2(R+\delta)< \mathrm{inj}(M)$. Then, their empirical measures
\[
\rho^{N, X}=\frac{1}{N}\sum_{i=1}^N\delta_{x_i}\quad\text{and}\quad \rho^{N, Y}=\frac{1}{N}\sum_{i=1}^N\delta_{y_i}
\]
satisfy
\[
\|u_\alpha(\rho^{N, X};x_i)-P_{x_iy_i}u_\alpha(\rho^{N, Y};y_i)\|\leq C_4 D(X, Y)\quad\forall~1\leq i\leq N,
\]
where $C_4:=C_2+C_3$.
\end{lemma}

\begin{proof}
The proof is provided in Appendix \ref{proof:lem: Lip-particle}

\end{proof}

In the analysis of the diffusion term, we need to compare tangent vectors belonging to different tangent spaces. Besides the Lipschitz continuity of the scalar diffusion coefficient, this requires a quantitative control of the variation of the chosen local orthonormal frame under parallel transport. The following lemma shows that after identifying tangent spaces by parallel transport, the frame changes at most linearly with respect to the geodesic distance between the base points.

\begin{lemma} \label{lem:frame}
Let $\{E_k\}$ be a smooth local orthonormal frame on $M$. Then, there exists a constant $C_E>0$ such that for all $x,y\in B_{R+\delta}(o)$, we have
\[
\|E_k(x) - P_{xy} E_k(y) \| \leq C_E d(x,y).
\]
\end{lemma}

\begin{proof}
Since $x,y\in B_{R+\delta}(\p)$, there exists the unique geodesic connecting $x$ and $y$. For $\ell :=d(x,y)$, choose a  unit speed geodesic $\gamma:[0,\ell]\to M $ such that
\[
\gamma(0)=x,\quad \gamma(\ell)=y.
\]
Define $V(s) := P_{x\gamma(s)}E_k(\gamma(s))\in T_{x}M$. Then, 
\[
V(0)= E_k(x),\quad V(\ell) = P_{xy}E_k(y)
\]
and in addition, there exists a constant $C_{E,k}$ such that
\[
V'(s) =  P_{x^1\gamma(s)} (\nabla_{\dot \gamma(s)} E_k),\quad \|V'(s)\| = \|\nabla_{\dot \gamma(s) }E_k\| \leq C_{E,k} \leq \max_{1\leq k \leq d}C_{E,k}=:C_E
\]
which yields the desired estimate
\begin{align*} 
\| E_k(x) - P_{xy} E_k(y) \| &= \|V(0) - V(\ell)\|  = \left\| -\int_0^\ell V'(s) \d s\right\| \leq C_{E} d(x,y).
\end{align*}

\end{proof}

Below, we state and provide the proof of the theorem for global well-posedness of  particle system \eqref{eq: particle}.

\begin{theorem} \label{thm:T3.1}
Suppose that initial data $x_i(0)\in M$ satisfy
\[
\bE\left[ \frac{1}{N}\sum_{i=1}^N\dist(x_i(0), \p)^2\right]<\infty.
\]
Then for any $T>0$, there exists a unique solution $x_i \in C([0,T]; M)$ for all $i\in [N]$ to the particle system \eqref{eq: particle}. Moreover, there exists a constant $C>0$ such that the averaged second moment satisfies
    \begin{equation*}
        \bE\left[ \frac{1}{N}\sum_{i=1}^N\dist(x_i(t), \p)^2\right]\leq  C(1+t^2),\quad t\in [0,T].
    \end{equation*}
\end{theorem}

\begin{proof}
\textbf{Step 1:  Well-posedness.}
We establish the existence and uniqueness of a strong solution to the particle system up to time $T>0$. Consider the joint configuration space $M^N$ equipped with the product metric $D(X, Y)$. To guarantee a unique strong solution without finite-time explosion, it follows from \cite[Theorem 1.2.9]{hsu2002stochastic} that we must verify that both the drift vector field and the diffusion tensor are locally Lipschitz continuous and satisfy a global growth bound.

$\bullet$ (Drift vector field): Let $b_i(X)$ be  the drift vector field for the $i$-th particle:
\[
b_i(X) := \lambda u_\alpha(\rho^{N, X}; x_i) \in T_{x_i}M.
\]
For any two configurations $X, Y \in M^N$ that are sufficiently close (i.e., $D(X, Y) \leq 2(R+\delta)< \mathrm{inj}(M)$), the unique shortest geodesic is well-defined. Using Lemma \ref{lem: Lip-particle}, the distance between drift vectors under parallel transport $P_{x_i y_i}$ from $y_i$ to $x_i$ satisfies:
\[
\|b_i(X) - P_{x_i y_i}b_i(Y)\| \leq \lambda C_4 D(X, Y).
\]
This ensures the combined drift vector field on $M^N$ is locally Lipschitz continuous.

$\bullet$ (Scalar diffusion coefficient): Let  $a_i(X)$ be the scalar diffusion coefficient of the $i$-th particle:
\[
a_i(X) := \sigma h_{R-\delta, R}(\dist(x_i, \p)) \|u_\alpha(\rho^{N, X}; x_i)\|.
\]
The full diffusion term for the $i$-th particle can be viewed as a linear map $\mathcal{S}_i(X) : \mathbb{R}^d \to T_{x_i}M$ acting on the standard basis $\{e_k\}_{k=1}^d$ of $\mathbb{R}^d$ such that $\mathcal{S}_i(X)e_k = a_i(X)E_k(x_i)$. 

To establish Lipschitz continuity, we must bound the Hilbert-Schmidt norm, denoted by $\|\cdot\|_{\textup{HS}}$, of the difference between the operators at $X$ and $Y$ under parallel transport:
\begin{align*}
\|\mathcal{S}_i(X) - P_{x_i y_i}\mathcal{S}_i(Y)\|_{\mathrm{HS}} & = \sqrt{ \sum_{k=1}^d \| \mathcal{S}_i(X)e_k - P_{x_i y_i}\mathcal{S}_i(Y)e_k \|^2 } \\
&\leq \sum_{k=1}^d \|a_i(X)E_k(x_i) - a_i(Y)P_{x_i y_i}E_k(y_i)\|.
\end{align*}
Adding and subtracting the cross term $a_i(Y)E_k(x_i)$ inside the norm and applying the triangle inequality yields:
\[
\|\mathcal{S}_i(X) - P_{x_i y_i}\mathcal{S}_i(Y)\|_{\mathrm{HS}} \leq \sum_{k=1}^d \Big( |a_i(X) - a_i(Y)| \|E_k(x_i)\| + |a_i(Y)| \|E_k(x_i) - P_{x_i y_i}E_k(y_i)\| \Big).
\]
We now bound each component:
\begin{itemize}
    \item Since $\{E_k\}$ is an orthonormal frame, $\|E_k(x_i)\| = 1$.
    \item By the smoothness of the chosen local orthonormal frame on $M$, for any two configurations $X,Y\in M^N$ that $D(X, Y) \leq 2(R+\delta)< \min\left(\mathrm{inj}(M), \frac{\pi}{\sqrt{\kappa_+}}\right)$ we recall from Lemma \ref{lem:frame} that  
    \[
    \|E_k(x_i) - P_{x_i y_i}E_k(y_i)\| \leq C_E \dist(x_i, y_i) \leq C_E D(X, Y).
    \]
    \item Because the cutoff function $h_{R-\delta, R}$ is smooth, compactly supported, and bounded by $1$, $a_i(X)$ is globally bounded by some constant $A_{\max} = 2\sigma (R+\delta)$ by using Lemma \ref{L4.1}. 
    \item Furthermore, $a_i(X)$ is locally Lipschitz continuous. Specifically, combining the Lipschitz constant $\textup{Lip}(h_{R-\delta,R})$ of the cutoff function $h_{R-\delta,R}$ and the locally Lipschitz constant $C_4$ of $u_\alpha$ from  Lemma \ref{lem: Lip-particle}, there exists a constant $L_a$ such that
    \[
    |a_i(X) - a_i(Y)| \leq L_a D(X,Y).
    \]
\end{itemize}
By substituting these bounds  into the inequality, we obtain:
\[
\|\mathcal{S}_i(X) - P_{x_i y_i}\mathcal{S}_i(Y)\|_{\mathrm{HS}} \leq \sum_{k=1}^d \Big( L_a D(X,Y) + A_{\max} C_E D(X,Y) \Big) = d(L_a + A_{\max}C_E) D(X,Y).
\]
This proves that the diffusion tensor is locally Lipschitz continuous on $M^N$.

Because the cutoff function \(h_{R-\delta,R}(\dist(x_i,\p))\) explicitly vanishes for \(\dist(x_i,\p)\geq R\), the diffusion coefficient is globally bounded. Moreover, the drift coefficient is bounded by the cutoff \(h_{R,R+\delta}\) together with Lemma \ref{L4.1}.
Hence both the drift and diffusion coefficients are globally bounded on \(M^N\). In addition, since the coefficients are locally Lipschitz continuous and globally bounded (which trivially satisfies the linear growth condition), standard existence and uniqueness theorems for SDEs on manifolds in \cite[Theorem 1.2.9]{hsu2002stochastic} guarantee that the system does not explode in finite time, yielding a unique strong global solution $X(t) \in C ([0, T]; M^N)$.\\

\noindent\textbf{Step 2: Moment Bound.} First, we recall from Lemma \ref{L4.1} that
\[
\|u_\alpha(\rho^{N,X};x_i)\|\leq 2(R+\delta).
\]
Now, we use $\tilde{\Psi}_\p(x)=\frac{1}{2}\dist(x, \p)^2$ to get
\[
\nabla_M \tilde{\Psi}_\p(x)=-\log_x\p.
\]
It follows from  the chain rule that  for a sufficiently smooth function $g$, 
\begin{equation} \label{chainrule}
\Delta_Mg(\dist(x, x^*))=g''(\dist(x, x^*))+g'(\dist(x, x^*))\Delta_M\dist(x, x^*).
\end{equation}
From the Laplacian comparison theorem with $g=\tilde\Psi$ in \eqref{chainrule}, we also get
\begin{align*}
\Delta_M\tilde{\Psi}_\p(x) &\leq 1+(\dm-1) \sqrt{\kappa_-} d(x,o) \coth{ \sqrt{\kappa_-} d(x,o) } \\
& \leq d+(d-1) \sqrt{\kappa_-}d(x,o),
\end{align*}
where we used $s\coth{s} \leq 1+s$ for $s\geq0$. Recall the generator $\mathcal{L}$ of the SDE, and apply it to $\tilde{\Psi}_\p(x_i)$:
\begin{equation*}
\mathcal{L}_t\tilde{\Psi}_\p(x_i)=\langle\lambda u_\alpha(\rho_t;x_i),\nabla_M\tilde{\Psi}_\p(x_i)\rangle+\frac{\sigma^2}{2}h_{R-\delta,R}(\dist(x_i, \p))^2\|u_\alpha(\rho_t;x_i)\|^2\Delta_M\tilde{\Psi}_\p(x_i).
\end{equation*}
The first term can be estimated as
\[
\langle\lambda u_\alpha(\rho_t;x_i),\nabla_M\tilde{\Psi}_\p(x_i)\rangle\leq 2\lambda(R+\delta)\dist(x_i, \p),
\]
and the second term can be simplified into
\[
\frac{\sigma^2}{2}h_{R-\delta,R}(\dist(x_i, \p))^2\|u_\alpha(\rho_t;x_i)\|^2\Delta_M\tilde{\Psi}_\p(x_i)\leq 2\sigma^2(R+\delta)^2\left(\dm+(\dm-1) \sqrt{\kp_-}\dist(x_i, \p)\right).
\]
Thus, the generator satisfies 
\[
\mathcal{L}_t\tilde{\Psi}_\p(x_i)\leq 2\sigma^2(R+\delta)^2 d + \left(2\lambda(R+\delta)+2\sigma^2 (R+\delta)^2 (d-1)\sqrt{\kp_-}\right) \dist(x_i,\p).
\]
Let us denote 
\[
m_2^N(t):=\bE\left[ \frac{1}{N}\sum_{i=1}^N\dist(x_i(t), \p)^2\right].
\]
We take the expectations to find
   \begin{align*}
       \frac{\d  m_2^N(t)}{\d t}&=\frac{1}{N}\sum_{i=1}^N 2\bE\left[\frac{\d  \tilde{\Psi}_\p(x_i(t))}{\d t}\right]= \frac{1}{N}\sum_{i=1}^N 2\bE\left[\mathcal{L}_t\tilde{\Psi}_\p(X_t^i)\right]\\
      \leq& 4\sigma^2(R+\delta)^2\dm+\frac{1}{N}\left(4\lambda(R+\delta)+4\sigma^2(R+\delta)^2(\dm-1)\sqrt{\kp_-}\right)\sum_{i=1}^N\bE[\dist(x_i,\p)]\\
    \leq&4\sigma^2(R+\delta)^2\dm+\left(4\lambda(R+\delta)+4\sigma^2(R+\delta)^2(\dm-1)\sqrt{\kp_-}\right)\sqrt{m_2^N(t)}\\
    \leq& C+C \sqrt{m_2^N(t)}.
   \end{align*}
   Therefore, the growth rate of $m_2$ is less than quadratic  and we get the desired result.
\end{proof}

Next, we show the well-posedness of the mean-field dynamics \eqref{eq: Mckean} through a fixed-point argument.

\begin{theorem}  \label{thm:T3.2}
Suppose that the initial data $\rho_0$ is supported on $B_R(\p)$. Then for any $T>0$, there exists a unique nonlinear process $\bar x \in C([0,T], M)$ satisfying \eqref{eq: Mckean} in the strong sense. 
\end{theorem}

\begin{proof}
For a fixed measure curve $\mu \in C([0,T], \mathcal{P}_2(M))$, we consider the following linear frozen-measure SDE:
\begin{equation*} 
\d x(t)=\lambda u_\alpha(\mu_t;x(t))\d t+\sigma h_{R-\delta, R}(\dist(x(t), \p))\|u_\alpha(\mu_t;x(t))\| \sum_{k=1}^d E_k(x(t))\d B^k(t),
\end{equation*}
with the initial data $x(0)=x_0 \in B_{R}(\p)$. Since the coefficients are locally Lipschitz continuous and globally bounded by the geometric cutoffs, standard SDE theory on manifolds guarantees a unique strong solution $x(t)$. Furthermore, this global boundedness trivially ensures that the process has a finite second moment, satisfying $\mathbb{E}[\sup_{t\in[0,T]}\dist(x(t),\p)^2]<\infty$.
Let us denote the law of $x(t)$ as $\rho_t$, and define the map $\Phi:~C([0,T],\calP_2(\calM))\to C([0,T],\calP_2(\calM))$ such that
\begin{equation*}
    \Phi(\mu):=\rho\,.
\end{equation*}
\textbf{Step 1: Well-definedness.} We  claim that $\Phi(\mu)\in \calC([0,T],\calP_2(\calM))$. From Itô's formula, we have for $0\leq t_1\leq t$,
\begin{align*}
&\frac{\d}{\d t}\bE[\tilde{\Psi}_{x(t_1)}(x(t))]\\
&=\bE[\mathcal{L}_t\tilde{\Psi}_{x(t_1)}(x(t))]\\
&=\bE\left[\langle\lambda u_\alpha^t(x),\nabla_M\tilde{\Psi}_{x(t_1)}(x)\rangle+\frac{\sigma^2}{2}h_{R-\delta, R}(\dist(x, \p))^2\|u_\alpha^t(x)\|^2\Delta_M\tilde{\Psi}_{x(t_1)}(x)\right].
\end{align*}
By the definition of $\tilde{\Psi}_{x(t_1)}(x)=\frac{1}{2}\dist(x, x(t_1))^2$, we know
\[
\nabla_M\tilde{\Psi}_{x(t_1)}(x)=-\log_{x}x(t_1).
\]
This implies
\[
\bE[\langle\lambda u_\alpha^t(x),\nabla_M\tilde{\Psi}_{x(t_1)}(x)\rangle]\leq \lambda\sup_{\tau\in[t_1, t]}\|u_\alpha^t\|\bE\left[\dist(x(t_1), x(t))\right].
\]
In addition, if we use the Laplacian comparison theorem in Lemma \ref{lem:laplace}, we get
\[
\Delta_M\tilde{\Psi}_{x(t_1)}(x)\leq 1+(\dm-1)\sqrt{\kappa_-}\dist(x, x(t_1))\coth(\sqrt{\kappa_-}\dist(x, x(t_1)))\leq C_4+C_5\dist(x, x(t_1)).
\]
It yields
\begin{align*}
\bE\left[\frac{\sigma^2}{2}h_{R-\delta, R}(\dist(x, \p))^2\|u_\alpha^t(x)\|^2\Delta_M\tilde{\Psi}_{x(t_1)}(x)\right]&\leq \frac{\sigma^2}{2}\sup_{\tau\in[t_1, t]}\|u_\alpha^t\|^2\bE\left[C_4+C_5\dist(x, x(t_1))\right].
\end{align*}
Hence, we get
\begin{align*}
&\frac{\d}{\d t}\bE[\tilde{\Psi}_{x(t_1)}(x(t))]\leq C_6\sqrt{\bE[\tilde{\Psi}_{x(t_1)}(x(t))]}+C_7
\end{align*}
for some positive constants, and it yields
\[
\bE[\dist(x(t_1), x(t_2))^2]=2\bE[\tilde{\Psi}_{x(t_1)}(x(t_2))]\leq C_8(t_2-t_1)+C_9(t_2-t_1)^2
\]
for any $t_2>t_1$. This concludes
\begin{equation*}
W_2(\rho_{t_1},\rho_{t_2})\leq  \left(\bE[d(x(t_1),x(t_2))^2]\right)^{1/2}\leq \left(C_8(t_2-t_1)+C_9(t_2-t_1)^2\right)^{1/2},\quad\forall~t_2>t_1.
\end{equation*}
Therefore, the map $\Phi$ is well-defined and $\Phi(\mu)\in C([0,T],\calP_2(\calM))$.\\

\noindent\textbf{Step 2: Positive invariance of $B_{R+\delta}(o)$.} For a solution $\bar x(t)$ to McKean--Vlasov process, we use Itô's formula to see that  the law $\rho_t = \textup{Law}(\bar x(t))$ satisfies  the following PDE:
\[
\partial_t\rho_t(x)=-\lambda\nabla_M\cdot(u_\alpha(\mu_t;x)\rho_t(x))+\frac{\sigma^2}{2}\Delta_M\left(h_{R-\delta, R}(\dist(x, \p))^2\|u_\alpha(\mu_t;x)\|^2\rho_t(x)\right)
\]
in the distributional sense, since the generator is given as \eqref{generator}. From the definition of the weak solution, we know
\begin{align*} 
\frac{\d}{\d t}\int_M \phi(x)\d\rho_t(x)&=\lambda\int_M \nabla_M \phi(x)\cdot u_\alpha(\mu_t;x)\d \rho_t(x) \\
&\quad +\frac{\sigma^2}{2}\int_Mh_{R-\delta, R}(\dist(x, \p))^2\|u_\alpha(\mu_t;x)\|^2\Delta_M\phi(x)\d\rho_t(x).
\end{align*} 
Now, we assume $\phi_\epsilon:M\to[0, 1]$ is a smooth radially symmetric function with respect to $\p$ satisfying
\[
\phi_\epsilon(x):= \phi_\epsilon(\dist(x, \p)):=\begin{cases}
1&\quad \dist(x, \p)\in[0, R+\delta],\\
0&\quad \dist(x, \p)\in[R+\delta+\epsilon, \infty).
\end{cases}
\]
Furthermore, we can choose the set of test functions to satisfy $\{\phi_\epsilon(\dist(\cdot, \p))\}\searrow\chi_{B_{R+\delta}(\p)}$ as $\epsilon\searrow0$.  Since $R+\delta<\frac{\pi}{2\sqrt{\kp_+}}$, we choose small $\epsilon$ to satisfy 
\[
R+\delta+\epsilon<\frac{\pi}{2\sqrt{\kappa_+}}.
\]
On the one hand, we have
\[
\frac{\sigma^2}{2}\int_Mh_{R-\delta, R}(\dist(x, \p))^2\|u_\alpha(\mu_t;x)\|^2\Delta_M\phi_\epsilon(x)\d\rho_t(x)=0
\]
since $\Delta_M\phi_\epsilon(x)\equiv0$ for $x\in B_{R+\delta}(\p)$ and $|u_\alpha(\mu_t;x)|\equiv0$ on $B_{R+\delta}(\p)^c$. On the other hand, $\nabla_M\phi_\epsilon(\dist(x, \p))\cdot u_\alpha^t(x)\geq0$ (see Appendix \ref{secA:8}) implies
\[
\lambda\int_M \nabla_M \phi(x)\cdot u_\alpha(\mu_t;x)\d \rho_t(x)\geq0.
\]
Combining the above results gives
\[
\frac{\d}{\d t}\int_M \phi_\epsilon(\dist(x, \p))\d\rho_t(x)\geq0,\quad\forall~\epsilon>0,~t>0.
\]

Now, we assume that the initial distribution $\rho_0$ is supported in $B_{R+\delta}(\p)$, i.e., $\int_{B_{R+\delta}(\p)}\d\rho_0(x)=1$. We also fix $t>0$ and $\epsilon>0$. Then, we know the following inequality holds:
\[
\int_{M}\phi_\epsilon(x)\d\rho_t(x)\geq\int_{M}\phi_\epsilon(x)\d\rho_\tau(x),\quad\forall~\tau\in(0, t).
\]
Since $\phi_\epsilon\in C_b(M)$, we use the initial condition and the dominated convergence theorem to get
\[
\lim_{\tau\to0+}\int_M \phi_\epsilon(x) \d\rho_\tau(x)=\int_M\phi_\epsilon(x) \d\rho_0(x)\geq\int_{B_{R+\delta}(\p)}\d\rho_0(x)=1.
\]
Thus, we know the following inequality holds for any $t>0$ and $\epsilon>0$:
\[
\int_{M}\phi_\epsilon(x)\d\rho_t(x)\geq1.
\]
Again, it follows from  the dominated convergence theorem that 
\[
1\geq \int_{B_{R+\delta}(\p)}\d\rho_t(x)=\lim_{\epsilon\to0+}\int_M\phi_\epsilon \d\rho_t(x)\geq1,\quad\forall~t>0.
\]
Finally, we can conclude that $B_{R+\delta}(\p)$ is the invariant set of the given system. For this reason, we know $x(t)$ lies in $B_{R+\delta}(o)$ almost surely. \\

\noindent\textbf{Step 3: Contraction. }Now, we show that $\Phi$ is actually a contraction. To this end, for $\mu^1,\mu^2\in C([0,T],\calP_2(\calM))$, we consider solutions $x^i (i=1,2)$  to the linear SDEs driven by the same Brownian motion:
\begin{equation*}
    \d x^i(t)=\lambda u_\alpha(\mu_t^i;x^i)\d t+\sigma h_{R-\delta, R}(\dist(x^i,\p))\|u_\alpha(\mu_t^i;x^i)\|\sum_{k=1}^dE_k(x^i)\d B^{k}(t).
\end{equation*}
Now, we set
\[
b^i(x;t):=\lambda u_\alpha(\mu_t^i;x),\quad a^i(x;t):=\sigma h_{R-\delta, R}(\dist(x,\p))\|u_\alpha(\mu_t^i;x)\|.
\]
Then, we have
\begin{align*}
\d x^i(t)&=b^i(x^i;t)\d t+a^i(x^i;t)\sum_{k=1}^d E_k(x^i)\d B^k(t),\quad i=1, 2.
\end{align*}
The process $(x^1,x^2)$ satisfies an SDE on $M\times M$:
\[
\d(x^1,x^2)=(b^1(x^1;t),b^2(x^2;t))\d t+\sum_{k=1}^\dm(a^1(x^1;t) E_k(x^1),a^2(x^2;t) E_k(x^2))\d B^k(t).
\]
For the handy notation, we use
\[
V^i_k :=a^i(x^i;t) E_k(x^i).
\]
Then, the generator corresponding to this process is
\[
\tilde{\mathbf{L}}_t\phi(x^1, x^2) : =\langle\nabla\phi, (b^1, b^2)\rangle+\frac12\sum_{k=1}^\dm \mathrm{Hess}\phi(x^1, x^2)[(V_k^1, V_k^2), (V_k^1,V_k^2)],\quad\forall~\phi\in C^\infty(M\times M).
\]
Now, apply Itô's formula to the process $(x^1,x^2)$ with $\Psi(x^1, x^2)=\frac{1}{2}\dist(x^1, x^2)^2$ to get 
\begin{align*}
\frac{\d}{\d t}\bE[\Psi(x^1(t), x^2(t))]
=&\bE[\tilde{\mathbf{L}}_t\Psi(x^1(t),x^2(t))]\\
=&\underbrace{\bE[\langle\nabla_{x^1}\Psi(x^1(t), x^2(t)), b^1\rangle+\langle\nabla_{x^2}\Psi(x^1(t), x^2(t)), b^2\rangle]}_{=:\mathcal I_1}\\
&+\underbrace{\frac{1}{2}\sum_{k=1}^\dm\bE\left[\mathrm{Hess}\Psi(x^1, x^2)[(V_k^1, V_k^2), (V_k^1, V_k^2)]\right]}_{=:\mathcal I_2}. 
\end{align*}
Below, we estimate $\mathcal I_1$ and $\mathcal I_2$, respectively. Since the solution remains in $B_{R+\delta}(\p)$ almost surely, we may consider two sample paths of the SDEs:
\[
x^1(t), x^2(t)\in B_{R+\delta}(\p)\quad\forall~ t\ge 0.
\]
$\bullet$ (Estimate of $\mathcal I_1$): We observe 
\[
\nabla_{x^1}\Psi(x^1, x^2)=-\log_{x^1}x^2,\quad \nabla_{x^2}\Psi(x^1, x^2)=-\log_{x^2}x^1.
\]
Then, $\mathcal I_1$  can be simplified as follows:
\begin{align*}
\mathcal I_1&=\bE[\langle\nabla_{x^1}\Psi(x^1(t), x^2(t)), b^1\rangle+\langle\nabla_{x^2}\Psi(x^1(t), x^2(t)), b^2\rangle] = \bE[-\langle\log_{x^1}x^2, b^1-P_{x^1x^2}b^2\rangle]\\
&\leq\bE[\dist(x^1, x^2)\|b^1-P_{x^1x^2}b^2\|] = \lambda\bE[\dist(x^1, x^2)\|u_\alpha(\mu_t^1;x^1)-P_{x^1x^2}u_\alpha(\mu_t^2;x^2)\|]\\
&\leq\lambda\bE[\dist(x^1, x^2)\|u_\alpha(\mu_t^1;x^1)-P_{x^1x^2}u_\alpha(\mu_t^1;x^2)\|] \\
&\quad +\lambda\bE[\dist(x^1, x^2) \|P_{x^1x^2}u_\alpha(\mu_t^1;x^2)-P_{x^1x^2}u_\alpha(\mu_t^2;x^2)\|]\\
&=: \mathcal I_{11} + \mathcal I_{12}
\end{align*}

$\diamond$ (Estimate of $\mathcal I_{11}$): We use Lemma \ref{L4.2} to find
\[
\mathcal I_{11} \leq \lambda C_2 \bE [d(x^1,x^2)^2].
\]

$\diamond$ (Estimate of $\mathcal I_{12}$): We observe
\begin{align*} 
\|P_{x^1x^2}u_\alpha(\mu_t^1;x^2)-P_{x^1x^2}u_\alpha(\mu_t^2;x^2)\| & = \| u_\alpha(\mu_t^1;x^2)-u_\alpha(\mu_t^2;x^2)\|] \\
& \leq  C_3 W_2(\mu_t^1,\mu_t^2),
\end{align*}
where the isometric property of parallel transport is used for the equality and Lemma \ref{L4.3} for the inequality. Thus, by H\"older's inequality, we have
\[
\mathcal I_{12} \leq \frac{\lambda C_3}{2} \bE[ d(x^1,x^2)^2] + \frac{\lambda C_3}{2} W_2(\mu_t^1,\mu_t^2)^2.
\]
Thus, there exist two positive constants $C_{10}$ and $C_{11}$ such that 
\[
\mathcal I_1 \leq C_{10} \bE[ d(x^1,x^2)^2] + C_{11} W_2(\mu_t^1,\mu_t^2)^2.
\]

\noindent $\bullet$ (Estimate of $\mathcal I_2$): It follows from estimate  \eqref{es: Hess} that 
\begin{align*} 
\mathrm{Hess}\Psi(x^1, x^2)[(V_k^1, V_k^2), (V_k^1, V_k^2)]&\leq \|V_k^1-P_{x^1x^2}V_k^2\|^2+\frac{\kappa_-\dist(x^1,x^2)^2}{2}(\|V_k^1\|^2+\|V_k^2\|^2) \\
& =: \mathcal I_{21}+ \mathcal I_{22}.
\end{align*}

$\diamond$ (Estimate of $\mathcal I_{21}$): We decompose $\|V_k^1-P_{x^1x^2}V_k^2\|$ into three terms:
\begin{align*}
&\|V_k^1-P_{x^1x^2}V_k^2\|\\
&=\|a^1(x^1;t)E_k(x^1)-a^2(x^2;t)P_{x^1x^2}E_k(x^2)\| \\
&=\|(a^1(x^1;t)-a^2(x^1;t))E_k(x^1)\|+\|(a^2(x^1;t)-a^2(x^2;t))E_k(x^1)\| \\
&\quad +\|a^2(x^2;t)(E_k(x^1)-P_{x^1x^2}E_k(x^2))\|\\
&\leq |a^1(x^1;t)-a^2(x^1;t)|+|a^2(x^1;t)-a^2(x^2;t)|+|a^2(x^2;t)|\cdot \|E_k(x^1)-P_{x^1x^2}E_k(x^2)\| \\
&= : \mathcal I_{211} + \mathcal I_{212} + \mathcal I_{213}.
\end{align*}

$\circ$ (Estimate of $\mathcal I_{211}$):  We use Lemma \ref{L4.3} to find 
\begin{align*}
\mathcal I_{211} & = |a^1(x^1;t)-a^2(x^1;t)|=\sigma\eta(\dist(x^1, \p))\left|\|u_\alpha(\mu_t^1;x^1)\|-\|u_\alpha(\mu_t^2;x^1)\|\right|\\
&\leq\sigma h_{R-\delta, R}(\dist(x^1, \p))\|u_\alpha(\mu_t^1;x^1)-u_\alpha(\mu_t^2;x^1)\|\\
&\leq \sigma C_{3}W_2(\mu_t^1,\mu_t^2).
\end{align*}

$\circ$ (Estimate of $\mathcal I_{212}$): We observe
\begin{align*}
\mathcal I_{212}&= |a^2(x^1;t)-a^2(x^2;t)|\\
&=\sigma|h_{R-\delta, R}(\dist(x^1, \p))\|u_\alpha(\mu_t^2;x^1)\|-h_{R-\delta, R}(\dist(x^2,\p))\|u_\alpha(\mu_t^2;x^2)\||\\
&\leq\sigma h_{R-\delta, R}(\dist(x^1,\p))\left|\|u_\alpha(\mu_t^2;x^1)\|-\|u_\alpha(\mu_t^2;x^2)\|\right| \\
&\quad +\sigma|h_{R-\delta, R}(\dist(x^1, \p))-h_{R-\delta, R}(\dist(x^2,\p))|\|u_\alpha(\mu_t^2;x^2)\|\\
&\leq\sigma\|u_\alpha(\mu_t^2;x^1)-u_\alpha(\mu_t^2;x^2)\|+\sigma|h_{R-\delta, R}(\dist(x^1, \p))-h_{R-\delta, R}(\dist(x^2,\p))|\|u_\alpha(\mu_t^2;x^2)\|\\
&\leq C_{12}\dist(x^1, x^2)
\end{align*}
where we used Lemma \ref{L4.2} to estimate the first term and used the Lipschitz property of $h_{R-\delta,R}$ and Lemma \ref{L4.1} to estimate the second term. 

$\circ$ (Estimate of $\mathcal I_{213}$): We recall from Lemma \ref{lem:frame} that 
\begin{align*} 
\| E_k(x^1) - P_{x^1x^2} E_k(x^2) \| &= \|V(0) - V(\ell)\|  = \left\| -\int_0^\ell V'(s) \d s\right\| \leq C_{E} d(x^1,x^2).
\end{align*}
Thus, we see
\[
\mathcal I_{213} \leq C_E |a^2(x^2;t)| d(x^1,x^2) \leq C_{13} d(x^1,x^2).
\]
We collect the estimates for $\mathcal I_{211}$, $\mathcal I_{212}$ and $\mathcal I_{213}$ to find
\[
\mathcal I_{21} = \|V_k^1-P_{x^1x^2}V_k^2\|^2 \leq  C_{14}W_2(\mu_t^1,\mu_t^2)^2+C_{15}\dist(x^1, x^2)^2.
\]

\noindent $\bullet$ (Estimate of $\mathcal I_{22}$): Recall from Lemma \ref{L4.1} that  $\|V_k^1\|^2$ and $\|V_k^2\|^2$ are bounded above by $4\sigma^2(R+\delta)^2$. Thus, we have
\[
\mathcal I_{22} = \frac{\kp_-d(x^1,x^2)^2}{2} (\|V_k^1\|^2 + \|V_k^2\|^2) \leq 4\kp_- \sigma^2(R+\delta)^2 d(x^1,x^2)^2
\]
and consequently, we see
\[
\bE\left[\mathrm{Hess}\Psi(x^1, x^2)[(V_k^1, V_k^2), (V_k^1, V_k^2)]\right]\leq C_{14}W_2(\mu_t^1,\mu_t^2)^2+C_{16}\bE[\dist(x^1, x^2)^2].
\]
Finally, we derive a differential inequality for $D(t):=\bE[\dist(x^1(t), x^2(t))^2]$:
\[
\frac{\d}{\d t}D(t)\leq C_{17}W_2(\mu_t^1,\mu_t^2)^2+C_{18}D(t).
\]
Recall that we have $D(0)=0$ since we consider the common initial condition of the two SDEs. 
By Grönwall's inequality, we integrate this to find:
\[
D(t) \leq C_{17} \int_0^t e^{C_{18}(t-s)} W_2(\mu_s^1,\mu_s^2)^2 \d s \leq C_{17} t e^{C_{18}t} \sup_{s\in [0,t]} W_2(\mu_s^1,\mu_s^2)^2.
\]
Taking the supremum over $t \in [0, T]$, we eventually obtain 
\[
\sup_{t \in [0,T]} W_2(\rho^1(t), \rho^2(t))^2\leq \sup_{t \in [0,T]} D(t)\leq C_{17} T e^{C_{18}T} \sup_{t \in [0,T]} W_2(\mu_t^1,\mu_t^2)^2.
\]
Therefore, if we choose a sufficiently small $T^*>0$, there exists $0<C_{19}<1$ such that
\[
\sup_{t \in [0,T^*]} W_2(\rho^1(t), \rho^2(t)) \leq C_{19} \sup_{t \in [0,T^*]} W_2(\mu^1_t,\mu_t^2).
\]
Then, $\Phi$ is a contraction and there exists a solution by the Banach fixed point theorem.

Because the constants $C_{17}$ and $C_{18}$ depend only on the geometry of $M$ and the cutoff parameters, and are completely independent of the initial measure $\rho_0$, we can iterate this contraction argument on sequential intervals $[T^*, 2T^*], [2T^*, 3T^*], \dots$ to extend the unique solution to any   time horizon $T>0$.
\end{proof}

In addition, it is easy to check that the associated law $\rho = \textup{law}(\bar x) \in C([0,T], \mathcal P_2(M))$ satisfies the corresponding Fokker--Planck equation \eqref{Main-eq} in the weak sense by using It\^{o}'s formula. 
Furthermore, if $x_0\in B_R(\p)$ and $\mu\in C([0,T],\mathcal{P}(B_R(\p)))$ then $\rho\in C([0,T],\mathcal{P}(B_R(\p)))$. Therefore, $B_R(\p)$ is an invariant set of the PDE in a similar way. If $B_R(\p)$ is an invariant set, we can reduce the system as follows:
\begin{align*}
\begin{cases}
&\displaystyle\partial_t\rho_t(x)=-\lambda\mathrm{div}_M\bigg(u_\alpha^t(x)\rho_t(x)\bigg)+\frac{\sigma^2}{2}\Delta_M\left(h_{R-\delta, R}\left(\dist(x, \p)\right)^2\|u_\alpha^t(x)\|^2\rho_t(x)\right),\vspace{0.3cm}\\
&\displaystyle u_\alpha^t(x)=\frac{1}{\int_{B_R(\p)}w_\alpha(y)\d \rho_t(y)}\int_{B_R(\p)}w_\alpha(y)\log_xy\d \rho_t(y).
\end{cases}
\end{align*}


\section{Global Convergence} \label{sec:4}
\setcounter{equation}{0}

In this section, we present our main result about the global convergence of the CBO dynamics in a mean-field law for the objective function $\mathcal E$.  In the previous section, the well-posedness of the mean-field dynamics was established. We now turn to the long-time behavior of the solution. The goal of this section is to show that for a sufficiently large $\alpha$, the law $\rho_t$ concentrates around a global minimizer $x^*$ of the objective function $\mathcal E$. The proof consists of four steps. \\

\noindent $\bullet$ Step A: First, we derive a dissipative estimate for the variance functional $\mathcal V[\rho_t]$. This estimate shows that $\mathcal V[\rho_t]$ decays if the velocity vector field $u_\alpha^t(x)$ is sufficiently close to the target field $\log_xx^*$ (see Lemma \ref{lem:dissipative}).  

\vspace{0.2cm}

\noindent $\bullet$ Step B: Second, we use a quantitative Laplace principle to estimate drift error $\|u_\alpha^t(x) - \log_xx^*\|$ (see Proposition \ref{prop: laplace}). 

\vspace{0.2cm}

\noindent $\bullet$ Step C: Third, we show that any initial mass located near the global minimizer persists for any finite time (see Proposition \ref{prop: positive mass}). 

\vspace{0.2cm}

\noindent $\bullet$ Step D: Finally, we combine these estimates through a stopping time argument to obtain exponential decay of $\mathcal V[\rho_t]$ up to the prescribed accuracy (see Theorem \ref{thm:convergence}).

\subsection{Dissipative estimate of the energy functional}
In the Euclidean space, the following functional was used:
\[
\mathcal{V}_{\textup{Euclidean}}[\rho] :=\frac{1}{2}\int_{\bbr^\dm}\|x-x^*\|^2\d\rho_t(x).
\]
Here, $x^*$ is the global minimizer of $\mathcal{E}$. Since $\|x^*-x\|$ is replaced by $\dist(x^*,x)$, we can naturally generalize the above functional as follows:
\[
\mathcal{V}[\rho] :=\frac{1}{2}\int_{B_R(\p)}\dist(x, x^*)^2\d\rho(x),
\]
where $x^*\in M$ is the global minimizer of $\mathcal{E}$. Here, we   integrated only over $B_R(\p)$ since we assume that $\rho_0$ is supported on $B_R(\p)$, which is positively invariant along PDE \eqref{Main-eq}. The functional $\mathcal V[\rho]$ measures the concentration of the law $\rho$ around the target minimizer $x^*$. In particular, smallness of $\mathcal V[\rho]$ implies that most of the probability mass is contained in a small geodesic neighborhood of $x^*$. In this regard, we derive the dissipative estimate of $\mathcal V[\rho]$.

\begin{lemma} \label{lem:dissipative}
Let $T>0$ and let $\rho\in C\left([0,T], \mathcal P_2(M)\right)$ be a weak solution to \eqref{Main-eq}. Then, the functional $\mathcal V[\rho]$ satisfies
\begin{align*}
\frac{\d}{\d t}\mathcal{V}[\rho_t] & \leq  -\left(2\lambda -   \frac{\sigma^2}{2} (\dm+2(d-1)\sqrt{\kappa_-}R) \right)\mathcal{V}[\rho_t]  \\
&\quad +\Big(  \lambda + \sigma^2 (\dm+2(d-1)\sqrt{\kappa_-}R)\Big)  \mathcal{V}[\rho_t]^{1/2}\sup_{x\in B_R(o)}\|u_\alpha^t(x)-\log_xx^*\| \\
&\quad +  \frac{\sigma^2}{2} (\dm+2(d-1)\sqrt{
\kappa_-}R) \sup_{x\in B_R(o)}\|u_\alpha^t(x)-\log_xx^*\|^2.
\end{align*}
\end{lemma}

\begin{proof}
We recall from Definition \ref{def:weaksol} the definition of a weak solution to see 
 \begin{align*}
\frac{\d}{\d t}\mathcal{V}[\rho_t]
&=\frac{1}{2}\frac{\d}{\d t}\int_{B_R(\p)} \dist(x, x^*)^2 \d \rho_t(x)\\
&=\frac{\lambda}{2}\int_{B_R(\p)} \left(\nabla_M\dist(x, x^*)^2\right)\cdot u_\alpha^t(x)\rho_t(x)\d x\\
&\qquad+\frac{\sigma^2}{4}\int_{B_R(\p)} \Delta_M(\dist(x, x^*)^2)h_{R-\delta, R}\left(\dist(x, \p)\right)^2 \|u_\alpha^t(x)\|^2 \rho_t(x) \d x\\
&=:\mathcal{I}_3+\mathcal{I}_4.
\end{align*}
We  estimate $\mathcal{I}_3$ and $\mathcal{I}_4$ one by one. \newline

\noindent $\bullet$ (Estimate of $\mathcal I_3$): We use H\"older's inequality to find 
\begin{align*}
\mathcal{I}_3&=-\lambda\int_{B_R(\p)} \log_xx^* \cdot u_\alpha^t(x) \d\rho_t(x)\\
&=-\lambda \int_{B_R(\p)} \log_xx^*\cdot \log_xx^*\d\rho_t(x)-\lambda\int_{B_R(\p)} \log_xx^*\cdot(u_\alpha^t(x)-\log_xx^*)\d\rho_t(x)\\
&\leq-2\lambda \mathcal{V}[\rho_t]+\lambda \int_{B_R(\p)}\dist(x, x^*)\|u_\alpha^t(x)-\log_xx^*\|\d\rho_t(x)\\
&\leq -2\lambda\mathcal{V}[\rho_t]+\lambda \mathcal{V}[\rho_t]^{1/2}\sup_{x\in B_R(o)}\|u_\alpha^t(x)-\log_xx^*\|.
\end{align*}

\noindent $\bullet$ (Estimate of $\mathcal I_4$): The second term can be estimated as
\begin{align*}
\mathcal{I}_4&=\frac{\sigma^2}{4}\int_{B_R(\p)} \Delta_M\left(\dist(x, x^*)^2\right)h_{R-\delta, R}\left(\dist(x, \p)\right)^2 \|u_\alpha^t(x)\|^2 \d\rho_t(x)\\
&\leq \frac{\sigma^2}{2}\int_{B_R(\p)} \bigg(1+(\dm-1)\sqrt{\kappa_-}\dist(x, x^*)\coth\left(\sqrt{\kappa_-}\dist(x, x^*)\right)\bigg)\|u_\alpha^t(x)\|^2\d\rho_t(x)\\
&\leq \frac{\sigma^2}{2}\int_{B_R(\p)}\left(\dm+(\dm-1)\sqrt{\kappa_-}\dist(x, x^*)\right)\|u_\alpha^t(x)\|^2 \d\rho_t(x).
\end{align*}
Here, we used $x\coth x\leq 1+x$ at the last inequality. Furthermore,  since $x, x^*\in {B_R(\p)}$, we have  $\dist(x, x^*)\leq 2R$. Thus, we can further estimate $\mathcal{I}_4$ as follows:
\begin{align*}
\mathcal{I}_4& \leq\frac{\sigma^2}{2} (\dm+2(d-1)\sqrt{\kappa_-}R) \int_{B_R(\p)}\|u_\alpha^t(x)\|^2 \rho_t(x) \d x \\
& = \frac{\sigma^2}{2} (\dm+2(d-1)\sqrt{\kappa_-}R) \int_{B_R(\p)}\|(u_\alpha^t(x)-\log_xx^*)+\log_xx^*\|^2\rho_t(x)\d x \\
& \leq  \frac{\sigma^2}{2} (\dm+2(d-1)\sqrt{\kappa_-}R) \\
&\quad \times \left(\sup_{x\in B_R(\p)}\|u_\alpha^t(x)-\log_xx^*\|^2+2\mathcal{V}[\rho_t]^{1/2}\sup_{x\in B_R(\p)}\|u_\alpha^t(x)-\log_xx^*\|+\mathcal{V}[\rho_t]\right).
\end{align*}
Hence, we obtain the desired result:
\begin{align*}
\frac{\d}{\d t}\mathcal{V}[\rho_t] &\leq -\left(2\lambda -   \frac{\sigma^2}{2} (\dm+2(d-1)\sqrt{\kappa_-}R) \right)\mathcal{V}[\rho_t]  \\
&\quad +\Big(  \lambda + \sigma^2 (\dm+2(d-1)\sqrt{\kappa_-}R)\Big)  \mathcal{V}[\rho_t]^{1/2}\sup_{x\in B_R(\p)}\|u_\alpha^t(x)-\log_xx^*\| \\
&\quad +  \frac{\sigma^2}{2} (\dm+2(d-1)\sqrt{
\kappa_-}R) \sup_{x\in B_R(\p)}\|u_\alpha^t(x)-\log_xx^*\|^2.
\end{align*}

\end{proof}

Lemma \ref{lem:dissipative} reduces the convergence problem to the control of the drift error
\[
\sup_{x\in B_R(o)}\|u_\alpha^t(x)-\log_xx^*\|.
\]
In fact, if this quantity can be bounded by a sufficiently small multiple of $\mathcal V[\rho_t]^\frac12$, then exponential decay of $\mathcal V[\rho_t]$ is derived. 
To obtain a quantitative estimate of the drift error, we impose the following structural assumptions on the objective function which are quite standard in the analysis of   CBO-type methods, for instance \cite{carrillo2018analytical,fornasier2024consensus}.\\

Besides \textbf{(E)} in the Introduction, we impose the following additional assumptions on $\mathcal{E}$.\\

\noindent\textbf{(A1)}: There exists a unique $x^*\in M$ such that $\mathcal E(x^*) = \inf_{x\in M} \mathcal E(x) =\underline{\mathcal E}$.

\noindent\textbf{(A2)}: there exist $\mathcal E_\infty, R_0,\eta>0$ and $\nu\in (0,\infty)$ such that
\begin{align*}
&\textup{(i)}~~d(x,x^*) \leq \frac{ (\mathcal E(x) - \underline{\mathcal E})^\nu}{\eta},\quad \forall~x\in B_{R_0}(x^*). \\
&\textup{(ii)}~~\mathcal E(x) -\underline{\mathcal E} >\mathcal E_\infty,\quad \forall~x\in (B_{R_0}(x^*))^c.
\end{align*} 

\noindent\textbf{(A3)}: choose $R_0>0$ and $\frac{ \mathcal E_\infty^\nu }{\eta}$ sufficiently small  in the sense that 
\[
\textup{(i)}~~B_{R_0}(x^*) \subseteq B_{R}(o). \quad \textup{(ii)}~~B_{\frac{ \mathcal E_\infty^\nu }{\eta}}(x^*) \subseteq B_{R}(o) .
\]

\begin{remark}
We provide some comments on (A3) for $\mathcal E$. First, regarding (A3)(i), we will later need to integrate over $B_r(x^*)$, for instance in Proposition \ref{prop: laplace}. For this purpose, it is necessary that $B_{R_0}(x^*)\subseteq B_R(o)$ for every $r\in (0,R_0]$. Thus, we need (A3)(i). Second,  for (A3)(ii), the domain of integration is $B_{\tilde r}(x^*)$, for instance in the proof of Proposition \ref{prop: laplace}. Again, since we will later assume that $\tilde r\leq \frac{ \mathcal E_\infty^\nu }{\eta} $, we need (A3)(ii) in order to ensure that the relevant integration domain is contained in $B_R(o)$.
\end{remark}

\subsection{Quantitative Laplace principle} 

In this subsection, we quantify how closely the drift $u_\alpha^t$ approximates the ideal drift toward the global minimizer $x^*$. The following proposition establishes a quantitative Laplace estimate for the drift. It shows that the drift $u_\alpha^t$ is uniformly close to the ideal field $x\mapsto \log_x x^*$ with a suitable error.

\begin{proposition}\label{prop: laplace}
Let $\alpha>0$ be fixed. For any $r>0$, define
\[
\mathcal E_r := \sup_{y\in B_r(x^*)} \mathcal E(y).
\]
Then, for any $ r \in (0, R_0]$ and $q>0$ such that 
\[
q+\mathcal E_r\leq \mathcal E_\infty,
\]
we have
\[
\sup_{x\in B_R(\p)}\|u_\alpha^t(x) - \log_xx^*\| \leq \frac{(q+\mathcal E_r)^\nu}{\eta} + \frac{e^{-\alpha q}}{ \int_{x\in B_r(x^*)} \rho_t(x) \d x   }\int_{B_R(o)} d(y,x^*)\rho(y)\d y.
\]
\end{proposition}


\begin{proof}

We fix $r>0$ first. Now, let us estimate $\sup_{x\in B_R(\p)}\|u_\alpha^t(x)-\log_xx^*\|$. To do so, we consider arbitrary $x\in B_R(\p)$ first. 

\begin{align*}
u_\alpha(\rho;x)-\log_xx^* &=\frac{1}{\|w_\alpha\rho\|_1}\int_{B_R(\p)}w_\alpha(y)\log_xy~\d\rho(y)-\log_x x^*\\
&=\frac{1}{\|w_\alpha \rho\|_1}\int_{B_R(o)} w_\alpha(y) (\log_xy-\log_xx^*)\d \rho(y).
\end{align*}
Thus, we have
\begin{align*}
\|u_\alpha(\rho;x)-\log_xx^*\|&\leq\frac{1}{\|w_\alpha \rho\|_1}\int_{B_R(o)} w_\alpha(y) \dist(y, x^*)\d \rho(y)=:\mathcal{I}_5.
\end{align*}
Here, we used $\|\log_x y-\log_x x^*\|\leq \dist(y, x^*)$. Let $\tilde r\geq r>0$. Then, we have
\begin{align*}
\mathcal I_{5} & = \frac{1}{\|w_\alpha   \rho\|_1}\int_{B_R(\p)\cap B_{\tilde{r}}(x^*)}w_\alpha(y)  \dist(y, x^*)\d \rho(y)+ \frac{1}{\|w_\alpha   \rho\|_1}\int_{B_R(\p)\cap B_{\tilde{r}}(x^*)^c}w_\alpha(y) \dist(y, x^*)\d \rho(y) \\
& = :\mathcal I_{51} +\mathcal I_{52}.
\end{align*}
Below, we estimate $\mathcal I_{51}$ and $\mathcal I_{52}$, respectively. \newline

\noindent $\bullet$ (Estimate of $\mathcal I_{51}$): If $y\in B_R(\p)\cap B_{\tilde r}(x^*)$, then $\dist(y, x^*)\leq \tilde{r}$. For this reason, we have
\begin{align*}
\mathcal I_{51} & = \frac{1}{\|w_\alpha   \rho\|_1}\int_{B_R(\p)\cap B_{\tilde{r}}(x^*)}w_\alpha(y) \dist(y, x^*)\d \rho(y) \leq \frac{\tilde{r}}{\|w_\alpha   \rho\|_1}\int_{B_R(\p)\cap B_{\tilde{r}}(x^*)}w_\alpha(y) \d\rho(y) \leq \tilde r.
\end{align*}

\noindent $\bullet$ (Estimate of $\mathcal I_{52}$): For any $a>0$, we observe
\begin{align*}
\|w_\alpha \rho\|_1 & = \int_{B_R(o)} w_\alpha(x)  \d\rho(x)  = \int_{B_R(o)} e^{-\alpha \mathcal E(x)}  \d\rho(x) \geq a \int_{ \{x\in B_R(o): e^{-\alpha \mathcal E(x)} \geq a\} }  \d\rho(x).
\end{align*}
Recall that
\[
\mathcal E_r = \sup_{y\in B_R(\p)\cap B_r(x^*)} \mathcal E(y).
\]
If we choose $a= e^{-\alpha \mathcal E_r}$, then 
\begin{align*}
\int_{ \{x\in B_R(\p): e^{-\alpha \mathcal E(x)}\geq e^{-\alpha \mathcal E_r}\} }  \d\rho(x)  = \int_{\{  x\in B_R(\p): \mathcal E(x) \leq \mathcal E_r  \} }   \d\rho(x)  \geq \int_{B_R(\p)\cap B_r(x^*)}   \d\rho(x) .  
\end{align*}
Finally, this gives the lower estimate of $\|w_\alpha   \rho\|_1$:
\[
\|w_\alpha \rho\|_1 \geq e^{-\alpha \mathcal E_r} \int_{B_R(\p)\cap B_r(x^*)} \d\rho(x) . 
\]
Now, we can estimate $\mathcal I_{52}$:
\begin{align*}
\mathcal I_{52} &= \frac{1}{\|w_\alpha   \rho\|_1}\int_ {B_R(\p)\cap B_{\tilde{r}}(x^*)^c}w_\alpha(y)  \dist(y, x^*)\d \rho(y) \\
&\leq \frac{e^{\alpha \mathcal{E}_r}}{\int_{B_R(\p)\cap B_r(x^*)}\d \rho(x)}\int_ {B_R(\p)\cap B_{\tilde{r}}(x^*)^c}w_\alpha(y)  \dist(y, x^*)\d \rho(y)\\
&\leq \frac{ \exp\left(-\alpha (  \inf_{y\in B_R(\p)\cap B_{\tilde r}(x^*)^c  } \mathcal E(y) - \mathcal E_r)\right)  }{  \int_{B_R(\p)\cap B_r(x^*)} \rho(x) \d x    } \int_{B_R(\p)\cap B_{\tilde{r}}(x^*)^c} d(y,x^*)  \d\rho(y) .
\end{align*}
Thus, for any $\tilde r\geq r>0$, 
\begin{align*}
\|u_\alpha^t(x)-\log_xx^* \|\leq \tilde r +  \frac{ \exp\left(-\alpha (  \inf_{y\in B_R(\p)\cap B_{\tilde r}(x^*)^c  } \mathcal E(y) - \mathcal E_r)\right)  }{  \int_{B_R(\p)\cap B_r(x^*)} \rho(x) \d x    } \int_{B_R(\p)\cap B_{\tilde{r}}(x^*)^c} d(y,x^*)  \d\rho(y) .
\end{align*}
Now we choose $\tilde r$ as
\[
\tilde r = \frac{(q+\mathcal E_r)^\nu}{\eta}
\]
where $\eta$ and $\nu$ are introduced in Assumption (A2) (i), $\mathcal E_r$ is defined in the proposition, and $q>0$ is an arbitrary number given in the proposition. Since $q+\mathcal E_r\leq \mathcal E_\infty$, 
\[
\tilde r \leq \frac{\mathcal E_\infty^\nu}{\eta}.
\]
On the other hand, we have to check that $\tilde r \geq r$ indeed holds. Since we consider $r\leq R_0$ and assume $\underline{\mathcal E}=0$ without loss of generality, we use Assumption (A2) (i) to find
\[
\tilde r = \frac{(q+\mathcal E_r)^\nu}{\eta} \geq \frac{\mathcal E_r^\nu}{\eta} = \frac{  (  \sup_{y\in B_r(x^*)} \mathcal E(y)      )^\nu    }{\eta} \geq \sup_{x\in B_r(x^*)} d(x,x^*)=r.
\]
Lastly, we use assumption (A2)(i) and (A2)(ii) to find 
\begin{align*}
  \inf_{y\in (B_{\tilde r}(x^*))^c ) } \mathcal E(y)   \geq \min\{\mathcal E_\infty, (\eta \tilde r)^\frac1\nu\} = (\eta\tilde r)^\frac1\nu
\end{align*}
depending on the sign of $\tilde r-R_0$. Thus, we have
\[
e^{-\alpha (  \inf_{y\in (B_{\tilde r}(x^*))^c ) } \mathcal E(y) - \mathcal E_r)} \leq e^{-\alpha q}.
\]
This gives the desired result:
\begin{align*}
\|u_\alpha^t(x)-\log_xx^* \| &\leq \frac{(q+\mathcal E_r)^\nu}{\eta}+ \frac{  e^{-\alpha (  \inf_{y\in (B_{\tilde r}(x^*))^c ) } \mathcal E(y) - \mathcal E_r) }}{  \int_{x\in B_r(x^*)} \rho(x) \d x    } \int_{(B_{\tilde{r}}(x^*))^c} d(y,x^*) \rho(y) \d y  \\
&\leq\frac{(q+\mathcal E_r)^\nu}{\eta} + \frac{  e^{-\alpha (  \inf_{y\in (B_{\tilde r}(x^*))^c ) } \mathcal E(y) - \mathcal E_r) }}{  \int_{x\in B_r(x^*)} \rho(x) \d x    } \int_{ B_R(o)} d(y,x^*) \rho(y) \d y\\
&\leq\frac{(q+\mathcal E_r)^\nu}{\eta} + \frac{  e^{-\alpha q }}{  \int_{x\in B_r(x^*)} \rho(x) \d x    } \int_{ B_R(o)} d(y,x^*) \rho(y) \d y.
\end{align*}
\end{proof}

Proposition \ref{prop: laplace} shows that the drift error can be made small in the large-$\alpha$ regime. However, it also has a possible obstruction: if the mass of $\rho_t$ near $x^*$ is small, then the smallness of the drift error is not guaranteed. Thus, we need to show that the dynamics does not instantaneously lose the initial mass around $x^*$. This is the purpose of the next subsection.

\subsection{Persistence of local mass near the global minimizer}

Next, we show that the probability mass $\int_{x\in B_r(x^*)} \rho(x) \d x $ has a positive lower bound for an arbitrarily small $r>0$. For this, we introduce a smooth mollifier $\phi_r: M \to [0,1]$ whose support is exactly $\overline{B_r(x^*)}$:
\begin{equation} \label{phi}
\phi_r(x) :=
\begin{cases}
\displaystyle e^{f(d(x,x^*))},\quad d(x,x^*)<r, \\
0,\quad \textup{otherwise}
\end{cases}, \quad f(s) := 1-\frac{r^2}{r^2 - s^2}.
\end{equation}
Then, we have
\[
\int_{x\in B_r(x^*)} \rho(x) \d x  \geq \int_M \phi_r(x) \rho_t(x) \d x .
\]
In order to obtain the desired positive lower bound, we estimate the right-hand side term. Before we move on, we need a  lower estimate on the Laplacian of $\phi_r$.

\begin{lemma} \label{lem:lower}
Let $\phi_r$ and $f$ be defined in \eqref{phi}. Then, we have
\begin{align*}
\Delta_M \phi_r(x)\geq\phi_r(x)\left(\frac{2r^2(3\dist(x, x^*)^4-r^4)}{(r^2-\dist(x, x^*)^2)^4}-\frac{2(\dm-1)r^2\sqrt{\kappa_-}\dist(x, x^*)}{(r^2-\dist(x, x^*)^2)^2}
\coth\left(\sqrt{\kappa_-}\dist(x, x^*)\right)\right).
\end{align*}
\end{lemma}

\begin{proof}
By a straightforward calculation, we see
\[
f'(s)=-\frac{2r^2s}{(r^2-s^2)^2},\quad f''(s)=-\frac{2r^2(r^2+3s^2)}{(r^2-s^2)^3}.
\]
For later use, we calculate $\nabla _x\phi_r(x)$:
\[
\nabla_x \phi_r(x) =  \frac{2r^2 \log_xx^*}{(r^2- d(x,x^*)^2)^2} \phi_r(x).
\]
Let $g$ be a function of the distance $\dist(x, x^*)$.  We recall the chain rule in \eqref{chainrule} 
\[
\Delta_Mg(\dist(x, x^*))=g''(\dist(x, x^*))+g'(\dist(x, x^*))\Delta_M\dist(x, x^*)
\]
and the Laplace comparison theorem for $\Delta_M\dist(x, x^*)$ in Lemma \ref{lem:laplace}:
\[
\Delta_M\dist(x, x^*)\leq (\dm-1)\sqrt{\kappa_-}\coth(\sqrt{\kappa_-}\dist(x, x^*)).
\]
If $g'\leq0$, then  we obtain the lower estimate of the Laplacian for $g(\dist (x,x^*))$:
\begin{align*}
&g''(\dist(x, x^*))+(\dm-1)\sqrt{\kappa_-}g'(\dist(x, x^*))\coth\left(\sqrt{\kappa_-}\dist(x, x^*)\right)\leq\Delta_M g(\dist(x, x^*))
\end{align*}
If we substitute $g(\dist(x, x^*))=\phi_r(x)=e^{f(\dist(x, x^*))}$, we have 
\begin{align*}
&g'(\dist(x, x^*))=f'(\dist(x, x^*))e^{f(\dist(x, x^*))}, \\
&g''(\dist(x, x^*))=\left(f'(\dist(x, x^*))^2+f''(\dist(x, x^*))\right)e^{f(\dist(x, x^*))}.    
\end{align*}
Finally, we use 
\[
f'(s)^2+f''(s)=\frac{2r^2(3s^4-r^4)}{(r^2-s^2)^4}
\]
to obtain the desired result:
\begin{align*}
\Delta_M \phi_r(x)&\geq\phi_r(x)\left(f'(\dist(x, x^*))^2+f''(\dist(x, x^*))+(\dm-1)\sqrt{\kappa_-}f'(\dist(x, x^*))\coth\left(\sqrt{\kappa_-}\dist(x, x^*)\right)\right) \\
& = \phi_r(x)\left(\frac{2r^2(3\dist(x, x^*)^4-r^4)}{(r^2-\dist(x, x^*)^2)^4}-\frac{2(\dm-1)r^2\sqrt{\kappa_-}\dist(x, x^*)}{(r^2-\dist(x, x^*)^2)^2} \coth\left(\sqrt{\kappa_-}\dist(x, x^*)\right)\right).
\end{align*}
\end{proof}

\begin{proposition}\label{prop: positive mass}
For all $t\in [0,T]$ and $r\in (0,R_0)$, we have
\begin{align*}
\int_{ x\in B_r(x^*)} \rho_t(x) \d x \geq  \left(\int_{x_\in B_r(x^*)} \rho_0(x) \d x\right) e^{-pt}
\end{align*}
for some $p>0$ given by
\[
p:= \max\left\{  p_1,p_2,\frac{2\lambda\beta}{m_r} \right\}.
\]
Here, $p_1$ and $p_2$ are defined in \eqref{p1p2}, $\beta$ is introduced in \eqref{K2} and satisfies  \eqref{const-beta}, and $m_r$ is defined in \eqref{mr}   for any $B>0$ with 
\[
\sup_{t\in [0,T]} \sup_{x\in B_R(\p)}\|u_\alpha^t(x) - \log_xx^*\|\leq B
\]
and for any $c\in (0,1)$ satisfying
\[
1+\sqrt{\kappa_-} r \leq \frac{3c^4-1}{2(d-1)(1-c)^2}.
\]


\end{proposition}

\begin{proof}
We recall
\[
\int_{x\in B_r(x^*)} \rho_t(x) \d x  \geq \int_M \phi_r(x) \rho_t(x) \d x .
\]
We use the formulation of a weak solution in Definition \ref{def:weaksol} to find
\begin{align*}
&\frac{\d}{\d t} \int_M \phi_r(x) \rho_t(x) \d x \\
&\quad = \int_M \phi_r(x) \partial_t \rho_t(x) \d x  \\
&\quad =\int_M   \phi_r(x) \left(   -\lambda \textup{div}_M (u_\alpha^t(x) \rho_t(x) ) + \frac{\sigma^2}{2} \Delta_M (  h_{R-\delta,R}(d(x,o))^2 \|u_\alpha^t(x)\|^2 \rho_t(x) )        \right) \d x \\ 
&\quad =   \int_M  \underbrace{  \lambda   \langle  \nabla \phi_r(x), u_\alpha^t(x) \rangle}_{=:T_1(x)} \rho_t(x) \d x        + \int_M \underbrace{ \frac{\sigma^2}{2}\Delta_M \phi_r(x) h_{R-\delta,R}\left(\dist(x, \p)\right)^2 \|u_\alpha^t(x)\|^2 }_{=: T_2(x)}\rho_t(x) \d x  \\
&\quad  = \int_M (T_1(x) + T_2(x) ) \rho_t(x) \d x .
\end{align*}
Our goal is to show that there exists $p>0$ such that 
\begin{equation} \label{T1T2}
T_1(x) + T_2(x) \geq -p\phi_r(x),\quad x\in M .
\end{equation}
Since the mollifier $\phi_r(x)$ vanishes outside the ball $\Omega_r := \{ x\in M : d(x,x^*) <r\}$, we restrict ourselves to $\Omega_r$.  From (A3)(i), we know that $B_{R_0}(x^*) \subseteq B_R(o)$. If $r\in (0,R_0)$, then
\[
d(x,x^*)<r <R_0 ~~ \Longrightarrow ~~x\in B_R(o) ~~ \Longrightarrow ~~d(x,o)<R.
\]
 We introduce two subsets of $\Omega_r$:
\begin{align*}
K_1 := \{ x\in M: d(x,x^*) >\sqrt{c}r\}
\end{align*}
and  
\begin{equation} \label{K2}
K_2 := \{ x\in M: - \lambda \langle \log_xx^*,u_\alpha^t(x)\rangle (r^2 - d(x,x^*)^2)^2 > \frac{r^2}{\beta}h_{R-\delta,R}^2 \|u_\alpha^t(x)\|^2 d(x,x^*)^2 \}
\end{equation}
where $\beta$ will be determined  later in \eqref{const-beta}. Then, we decompose $\Omega_r$ into three subsets
\[
\Omega_r = (K_1^c \cap \Omega_r) \cup (K_1\cap K_2^c \cap\Omega_r) \cup ( K_1\cap K_2\cap \Omega_r).
\]
Below,  we provide the desired estimate \eqref{T1T2} on each subset. \newline

\noindent $\bullet$ Case (i) (Subset $K_1^c \cap \Omega_r)$: In this subset, we have
\[
d(x,x^*)<\sqrt c r,\quad x\in K_1^c.
\]
For $T_1(x)$, we have
\begin{align*}
T_1(x)  & =  \lambda   \langle  \nabla \phi_r(x), u_\alpha^t(x) \rangle = \lambda \left\langle  \frac{2r^2 \log_xx^*}{ (r^2- d(x,x^*)^2)^2} \phi_r(x), u_\alpha^t(x)\right\rangle \\
& \geq -2r^2 \lambda \frac{  d(x,x^*) \|u_\alpha^t(x)\|      } {      (r^2- d(x,x^*)^2)^2   }  \phi_r(x)   \\
& \geq -\frac{ 2\lambda\sqrt c (B+\sqrt cr)}{(1-c)^2r} \phi_r(x) =: -p_1 \phi_r(x)
\end{align*}
where we used
\[
\|u_\alpha^t(x) \| \leq \| u_\alpha^t(x) - \log_xx^*\| + d(x,x^*) \leq B + \sqrt cr
\]
and
\[
\frac{1}{ r^2- d(x,x^*)^2  } <\frac{1}{(1-c)r^2}.
\]
On the other hand for $T_2(x)$,  we use the lower estimate of $\Delta_M \phi_r(x)$ in Lemma \ref{lem:lower} and the upper estimate  $h\leq1$ to find
\begin{align*}
&T_2(x)  = \frac{\sigma^2}{2}\Delta_M \phi_r(x)  h_{R-\delta,R}\left(\dist(x, \p)\right)^2 \|u_\alpha^t(x)\|^2 \\
& \geq \frac{\sigma^2}{2} \left(  \frac{2r^2(3d(x,x*)^4 - r^4)}{(r^2 - d(x,x*)^2)^4} - \frac{   2(d-1)r^2\sqrt{\kappa_-}d(x,x*)  }{(r^2 - d(x,x*)^2)^2}  \coth(\sqrt{\kappa_-} d(x,x*))\right) \phi_r(x) \|u_\alpha^t(x)\|^2 \\
&\geq -\frac{\sigma^2}{2}  \left(  \frac{2r^6}{ (r^2 - d(x,x^*)^2)^4} + \frac{2r^2 (d-1)}{ (r^2 - d(x,x^*)^2)^2} (1+\sqrt{\kappa_-}d(x,x^*)     \right) \phi_r(x) \cdot   \|u_\alpha^t(x)\|^2 \\
& \geq -\frac{\sigma^2}{2} \left(   \frac{2r^6}{(1-c)^4r^4}  + \frac{2r^2(d-1)}{(1-c)^2 r^2}(1+\sqrt{\kappa_-}\sqrt cr)     \right)   (B+\sqrt cr)^2 \phi_r(x) \\
& = : -p_2\phi_r(x).
\end{align*}
Hence in $K_1^c \cap \Omega_r$, we have
\[
T_1(x) + T_2(x) \geq -\max \{p_1,p_2\}\phi_r(x)
\]
where $p_1$ and $p_2$ are defined as
\begin{equation} \label{p1p2}
p_1= \frac{ 2\lambda\sqrt c (B+\sqrt cr)}{(1-c)^2r},\quad p_2 = \frac{\sigma^2}{2} \left(   \frac{2r^6}{(1-c)^4r^4}  + \frac{2r^2(d-1)}{(1-c)^2 r^2}(1+\sqrt{\kappa_-}\sqrt cr)     \right)   (B+\sqrt cr)^2 .
\end{equation}

\vspace{0.3cm}

\noindent $\bullet$ Case (ii) (Subset $K_1\cap K_2^c \cap \Omega_r$):  In this subset, we have
\[
d(x,x^*)>\sqrt cr 
\]
and  
\[
\lambda \langle \log_xx^*,u_\alpha^t(x)\rangle (r^2 - d(x,x^*)^2)^2 < \frac{r^2}{\beta} \|u_\alpha^t(x)\|^2 d(x,x^*)^2 .
\]
Our goal is now to show that 
\[
T_1(x) +T_2(x) \geq0, \quad x \in K_1\cap K_2^c \cap \Omega_r.
\]
 We observe
\begin{align*}
\frac{T_1(x) + T_2(x)}{2r^2\phi_r(x)}& =   \frac{  \lambda \langle \log_xx^*,u_\alpha^t(x)\rangle }{  (r^2- d(x,x^*)^2)^2 } + \frac{\sigma^2}{2} \Delta_M \phi_r(x) h_{R-\delta,R}\left(\dist(x, \p)\right)^2 \|u_\alpha^t(x)\|^2 \cdot \frac{1}{2r^2\phi_r(x)}.
\end{align*}
It follows from the lower estimate of $\Delta_{M}\phi_r(x)$ that 
\begin{align*}
 \frac{\Delta_M \phi_r(x)}{2r^2 \phi_r(x)} &\geq \frac{  3d(x,x^*)^4-r^4 }{   (r^2 - d(x,x^*)^2)^4 }    - \frac{  d-1  }{  (r^2-d(x,x^*)^2)^2} (1+\sqrt{\kappa_-} d(x,x^*)) \\
& = \frac{     (3d(x,x^*)^4-r^4  )  -(d-1)(1+\sqrt{\kappa_-}d(x,x^*)) (r^2 - d(x,x^*)^2)^2  }{    ((r^2 - d(x,x^*)^2)^4       } .
\end{align*}
Hence, we have
\begin{align*}
\frac{T_1(x) + T_2(x)}{2r^2\phi_r(x)} & \geq \frac{  \lambda \langle \log_xx^*,u_\alpha^t(x)\rangle (r^2- d(x,x^*)^2)^2 }{  (r^2- d(x,x^*)^2)^4 } \\
&\quad + \frac{\sigma^2  }{2} h_{R-\delta,R}\left(\dist(x, \p)\right)^2 \|u_\alpha^t(x)\|^2  \\
&\quad\quad \times \frac{     (3d(x,x^*)^4-r^4  )  -(d-1)(1+\sqrt{\kappa_-}d(x,x^*) (r^2 - d(x,x^*)^2)^2  }{  (  (r^2 - d(x,x^*)^2)^4     }  .
\end{align*}
In order to verify $T_1(x)+T_2(x)\geq0$, it suffices to show that
\begin{align*}
&\left(     \lambda \langle \log_xx^*,u_\alpha^t(x)\rangle   - \frac{\sigma^2  (d-1)}{2} h_{R-\delta,R}^2\|u_\alpha^t(x)\|^2 (1+\sqrt{\kappa_-} d(x,x^*) )\right) (r^2 - d(x,x^*)^2)^2\\
& \quad + \frac{\sigma^2   }{2}  h_{R-\delta,R}\left(\dist(x, \p)\right)^2\|u_\alpha^t(x)\|^2   (3d(x,x^*)^4-r^4  )  \geq 0
\end{align*}
which is equivalent to 
\begin{align*}
&\left(    -\lambda \langle \log_xx^*,u_\alpha^t(x)\rangle   + \frac{\sigma^2   (d-1)}{2} h_{R-\delta,R}^2\|u_\alpha^t(x)\|^2 (1+\sqrt{\kappa_-} d(x,x^*) )\right) (r^2 - d(x,x^*)^2)^2 \\
&\leq  \frac{\sigma^2    }{2}  h_{R-\delta,R}\left(\dist(x, \p)\right)^2\|u_\alpha^t(x)\|^2   (3d(x,x^*)^4-r^4  ).
\end{align*}
Our goal is to verify  that
\begin{equation} \label{case2-1}
 -\lambda \langle \log_xx^*,u_\alpha^t(x)\rangle(r^2 - d(x,x^*)^2)^2  \leq \frac{\sigma^2   }{4} h_{R-\delta,R}^2 \|u_\alpha^t(x)\|^2   (3d(x,x^*)^4-r^4  )
\end{equation}
and
\begin{align} \label{case2-2}
\begin{aligned}
  &\frac{\sigma^2   (d-1)}{2}h_{R-\delta,R}^2\|u_\alpha^t(x)\|^2 (1+\sqrt{\kappa_-} d(x,x^*) )(r^2 - d(x,x^*)^2)^2 \\
  &\quad \leq \frac{\sigma^2   }{4}h_{R-\delta,R}^2 \|u_\alpha^t(x)\|^2   (3d(x,x^*)^4-r^4  ).
\end{aligned}
\end{align}
 By adding these two inequalities \eqref{case2-1} and \eqref{case2-2}, we conclude that $T_1(x)+T_2(x)\geq0$. For the first inequality \eqref{case2-1}, we assume that $\beta$ introduced in subset $K_2$ satisfies
\begin{equation} \label{const-beta}
\frac{1}{c\beta} <\frac{\sigma^2  }{4}\left(3-\frac{1}{c^2}\right).
\end{equation}
Then, we have
\begin{align*}
 -\lambda \langle \log_xx^*,u_\alpha^t(x)\rangle(r^2 - d(x,x^*)^2)^2 &<  \frac{r^2}{\beta}h_{R-\delta,R}^2 \|u_\alpha^t(x)\|^2 d(x,x^*)^2 \\
 &< \frac{1}{c\beta}h_{R-\delta,R}^2 \|u_\alpha^t(x)\|^2 d(x,x^*)^4 \\
 &<\frac{\sigma^2  }{4}\left(3-\frac{1}{c^2}\right)h_{R-\delta,R}^2 \|u_\alpha^t(x)\|^2 d(x,x^*)^4  \\
 & < \frac{\sigma^2  }{4} h_{R-\delta,R}^2 \|u_\alpha^t(x)\|^2 (3d(x,x^*)^4 - r^4).
\end{align*}
Next, for the second inequality, it is equivalent to 
\[
(d-1)( 1+ \sqrt{\kappa_-}d(x,x^*)) (r^2 - d(x,x^*)^2)^2 \leq \frac{1}{2}  (3d(x,x^*)^4-r^4  ).
\]
Define an auxiliary polynomial
\[
g(s) := 3s^4 - r^4 - 2(d-1)(1+\sqrt{\kappa_-}s)(r^2- s^2)^2.
\]
Then, our goal is to show that $g(s)\geq0$. Since $\sqrt cr\leq d(x,x^*)=s\leq r$, we have
\begin{align*}
g(s) &\geq 3c^4 r^4-r^4 - 2(d-1)(1+\sqrt{\kappa_-}r)(1-c)^2 r^4 \\
&= ( 3c^4 -1 - 2(d-1)(1+\sqrt{\kappa_-}r))r^4 \geq0
\end{align*}
whenever
\[
1+\sqrt{\kappa_-} r \leq \frac{3c^4-1}{2(d-1)(1-c)^2}.
\]
Hence, for any $r>0$, we choose $c$ sufficiently close to 1 so that the above relation is satisfied. \newline

\noindent $\bullet$ Case (iii) (Subset $K_1 \cap K_2 \cap \Omega_r$): In this subset, we have
\[
d(x,x^*)>\sqrt cr
\]
and
\[
-\lambda \langle \log_xx^*,u_\alpha^t(x)\rangle (r^2 - d(x,x^*)^2)^2 > \frac{r^2}{\beta} h_{R-\delta,R}^2\|u_\alpha^t(x)\|^2 d(x,x^*)^2 .
\]
For $T_1(x)$, we recall the definition of $K_2$:
\begin{align*}
 \frac{   \langle \log_xx^*,u_\alpha^t(x)\rangle }{  (r^2- d(x,x^*)^2)^2 }  &\geq -\frac{ \| \log_xx^* \| \| u_\alpha^t(x) \|  }{  (r^2- d(x,x^*)^2)^2  } \geq  \frac{\lambda \beta}{r^2} \frac{  \langle \log_xx^*, u_\alpha^t(x)\rangle}{ \|\log_xx^*\| \|u_\alpha^t(x)\| h_{R-\delta,R}^2} \geq -\frac{\lambda \beta}{r^2 h_{R-\delta,R}^2}
\end{align*} 
Define
\[
D_r:= \sup_{x\in B_r(x^*)} d(x,o).
\]
Then, we claim that 
\[
D_r <R.
\]
Since $r<R_0$, we have $\overline{B_r(x^*)} \subseteq B_{R_0}(x^*) \subseteq B_R(o)$. Since $x\mapsto d(x,o)$ is continuous, it attains the maximum on $\overline{B_r(x^*)}$. Thus, there exists $x_{r_0} \in \overline{B_r(x^*)}$ such that
\[
d(x_{r_0},o) = \max_{x\in \overline{B_r(x^*)}} d(x,o).
\]
In addition, since $d(x_{r_0},o)<R$, we have
\[
D_r = \sup_{x\in \overline{B_r(x^*)}} d(x,o) \leq \max_{x\in \overline{B_r(x^*)}} d(x,o)  <R.
\]
This shows that since $d(x,o)\leq D_r<R$,  there exists $m_r>0$ such that 
\begin{equation} \label{mr}
h_{R-\delta,R}(d(x,o))^2 \geq m_r,\quad d(x,o)\leq D_r<R.
\end{equation}
Thus, we have
\begin{align*}
T_1(x)& =  \frac{ \lambda \langle \log_xx^*,u_\alpha^t(x)\rangle }{  (r^2- d(x,x^*)^2)^2 } 2r^2\phi_r(x)  \geq -2\frac{\lambda \beta}{m_r} \phi_r(x).
\end{align*}
Next, for $T_2(x)$, to see that $T_2(x)\geq0$, it suffices to show that
\[
 (3d(x,x^*)^4-r^4  )  -(d-1)(1+\sqrt{\kp_-}d(x,x^*)) (r^2 - d(x,x^*)^2)^2  \geq0 .
 \]
 In fact, in Case (ii), we show that
 \begin{align*}
 (3d(x,x^*)^4-r^4  ) &\geq 2(d-1)(1+\sqrt{\kappa_-}d(x,x^*) )(r^2 - d(x,x^*)^2)^2  \\
 &\geq (d-1)(1+\sqrt{\kappa_-}d(x,x^*)) (r^2 - d(x,x^*)^2)^2.
 \end{align*}
Thus, we get
\[
T_2(x)\geq0.
\]
Hence, we have
\[
T_1(x) + T_2(x) \geq -\max\left\{ p_1,p_2, \frac{2\lambda \beta}{m_r}\right\} \phi_r(x) =-p\phi_r(x)
\]
and this gives
\[
\frac{\d}{\d t} \int_M \phi_r(x) \rho_t(x) \d x  = \int_M (T_1(x) + T_2(x) ) \rho_t(x) \d x  \geq -p\int_M  \phi_r(x) \rho_t(x) \d x.
\]
Finally, Gr\"onwall's inequality gives the desired result. 
\end{proof}

By combining Propositions \ref{prop: laplace} and \ref{prop: positive mass}, we obtain the key estimates behind global convergence. Precisely, Proposition \ref{prop: positive mass} gives a positive lower bound on $\rho_t(B_r(x^*))$ on any finite time interval, whereas Proposition \ref{prop: laplace} shows that once this lower bound is available, the drift error can be  arbitrarily small by choosing $\alpha$ sufficiently large.

\subsection{Proof of the global convergence toward a minimizer}
We now state the global convergence result in the mean-field regime. The theorem states that if the attraction strength is sufficiently large compared to the diffusion-curvature contribution and $\alpha$ is chosen large enough, then the variance functional $\mathcal V[\rho_t]$ decays exponentially until the law reaches an arbitrarily small neighborhood of the global minimizer.

\begin{theorem}\label{thm:convergence}
Let $\rho \in C\bigl([0,T], \mathcal{P}_2(M)\bigr)$ be a weak solution to equation~\eqref{Main-eq} on any time interval $[0,T]$ with $T>0$.  In addition, the initial datum satisfies
\[
\int_{x\in B_r(x^*)} \rho_0(x) dx >0,\quad \forall~r\in (0,R_0).
\]
Then, for every $\varepsilon\in (0,\mathcal V[\rho_0])$, there is $\alpha_0>0$ such that, for each $\alpha\ge \alpha_0$, one can find a time 
\[
T_* \leq T_\varepsilon := -\frac{2}{C_{20}} \log\!\left( \frac{\varepsilon}{\mathcal{V}[\rho_0]} \right),
\quad \text{where} \quad
C_{20} := 2\lambda - \frac{\sigma^2}{2} \bigl(d + 2(d-1)\sqrt{\kappa_-}\, R\bigr) > 0,
\]
satisfying
\[
\mathcal{V}[\rho_{T_*}] \leq \varepsilon,
\]
and such that for all $t \in [0, T_*]$,
\[
\mathcal{V}[\rho_t] \leq e^{-\frac{C_{20}}{2} t} \, \mathcal{V}[\rho_0].
\]
\end{theorem}
\begin{proof}
We begin by recalling the differential inequality from Lemma~\ref{lem:dissipative}:
\begin{align}\label{eq: calV}
\frac{\mathrm{d}}{\mathrm{d} t}\mathcal{V}[\rho_t]
 \leq  -C_{20}\mathcal{V}[\rho_t]  
      + C_{21}\mathcal{V}[\rho_t]^{1/2}\sup_{x\in M}\|u_\alpha^t(x)-\log_x x^*\|  
      + C_{22} \sup_{x\in M}\|u_\alpha^t(x)-\log_x x^*\|^2,
\end{align}
where the constants are defined as
\begin{align*}
   C_{21} := \lambda + \sigma^2 \bigl(d+2(d-1)\sqrt{\kappa_-}\, R\bigr), \quad  C_{22} := \frac{\sigma^2}{2} \bigl(d+2(d-1)\sqrt{\kappa_-}\, R\bigr).
\end{align*}
For any $\varepsilon > 0$, define the stopping times
\[
T^\alpha_1 := \inf\bigl\{t > 0 \mid \mathcal{V}[\rho_t] \leq \varepsilon \bigr\} \in \mathbb{R} \cup \{+\infty\},
\]
\[
T^\alpha_2 :=\inf\Bigl\{t > 0 \,\Big|\, \sup_{x\in M}\|u_\alpha^t(x) - \log_x x^*\| \geq \mathcal{V}[\rho_0]^{1/2} + 3 \Bigr\} \in \mathbb{R} \cup \{+\infty\},
\]
and set $T_\alpha := \min\{T^\alpha_1, T^\alpha_2, T_\varepsilon\}$.

We will show that for sufficiently large $\alpha$, the following two estimates hold for all $t \in [0, T_\alpha)$:
\begin{align}\label{goal1}
    \sup_{x\in B_R(\p)}\|u_\alpha^t(x) - \log_x x^*\| 
    &\leq \min\!\left\{\frac{C_{20}}{4C_{21}},\, \sqrt{\frac{C_{20}}{4C_{22}}}\right\} \mathcal{V}[\rho_t]^{1/2}
    =: C_* \mathcal{V}[\rho_t]^{1/2},
\\[1ex]\label{goal2}
    \sup_{x\in B_R(\p)}\|u_\alpha^t(x) - \log_x x^*\| 
    &\leq \mathcal{V}[\rho_t]^{1/2} + 1.
\end{align}

We first  assume a priori that \eqref{goal1} and \eqref{goal2} hold. Assuming \eqref{goal1}, we substitute it into \eqref{eq: calV} to obtain, for all $t \in [0, T_\alpha]$,
\begin{align*}
\frac{\mathrm{d}}{\mathrm{d} t}\mathcal{V}[\rho_t]
&\leq -C_{20}\mathcal{V}[\rho_t] 
   + C_{21}\mathcal{V}[\rho_t]^{1/2} \cdot C_* \mathcal{V}[\rho_t]^{1/2}
   + C_{22} \bigl(C_* \mathcal{V}[\rho_t]^{1/2}\bigr)^2 \notag\\
&= -C_{20}\mathcal{V}[\rho_t] + C_{21} C_* \mathcal{V}[\rho_t] + C_{22} C_*^2 \mathcal{V}[\rho_t] \notag\\
&\leq -C_{20}\mathcal{V}[\rho_t] + \frac{C_{20}}{4}\mathcal{V}[\rho_t] + \frac{C_{20}}{4}\mathcal{V}[\rho_t]
 = -\frac{C_{20}}{2}\mathcal{V}[\rho_t],
\end{align*}
where the last inequality follows from the definition of $C_*$. Grönwall’s inequality then yields
\begin{equation*}
   \mathcal{V}[\rho_t] \leq e^{-\frac{C_{20}}{2}t} \mathcal{V}[\rho_0].
\end{equation*}
In particular, $\mathcal{V}[\rho_t]$ is strictly decreasing on $[0, T_\alpha]$. Consequently, by \eqref{goal2} and the continuity of $\mathcal{V}[\rho_t]$,
\begin{equation*}
    \sup_{x\in B_R(\p)}\|u_\alpha^{T_\alpha}(x) - \log_x x^*\|
    \leq \mathcal{V}[\rho_{T_\alpha}]^{1/2} + 1
    \leq \mathcal{V}[\rho_0]^{1/2} + 1.
\end{equation*}
By continuity of $\sup_{x\in M}\|u_\alpha^t(x) - \log_x x^*\|$, this implies $T^\alpha_2 > T_\alpha$. Hence,
\[
T_\alpha = \min\{T^\alpha_1, T_\varepsilon\},
\]
and therefore $\mathcal{V}[\rho_{T_\alpha}] \leq \varepsilon$, which completes the main part of the proof.

It remains to establish the claims \eqref{goal1} and \eqref{goal2}. By the definition of $T^\alpha_2$, the quantity $\sup_{x\in M}\|u_\alpha^t(x) - \log_x x^*\|$ is uniformly bounded for $t \in [0, T_\alpha)$. Thus, we may apply Proposition~\ref{prop: positive mass} with some constant $p > 0$ independent of $\alpha$ to obtain, for all $t \in [0, T_\alpha)$,
\begin{equation*}
    \rho_t\bigl(B_r(x^*)\bigr) \geq \rho_0\bigl(B_r(x^*)\bigr) e^{-pt} \geq \rho_0\bigl(B_r(x^*)\bigr) e^{-p T_\varepsilon} =: C_0 > 0.
\end{equation*}
We now choose parameters $q, r, \alpha  > 0$ such that
\begin{equation*}
    \underbrace{\frac{(q+\mathcal{E}_r)^\nu}{\eta}}_{\text{for } q,r \text{ small}} \leq \frac{C_*}{2}\varepsilon^{1/2}, \quad
    \underbrace{\frac{e^{-\alpha q}}{C_0}\sqrt{2}}_{\text{for } \alpha \text{ large}} \leq \frac{C_*}{2}.
\end{equation*}
With these choices and by Proposition~\ref{prop: laplace}, we have for all $t \in [0, T_\alpha)$,
\begin{align*}
    &\sup_{x\in B_R(\p)}\|u_\alpha^t(x) - \log_x x^*\| \notag\\
    &\qquad \leq \frac{(q+\mathcal{E}_r)^\nu}{\eta} 
               + \frac{e^{-\alpha q}}{ \int_{B_r(x^*)} \rho_t(x)\,\mathrm{d}x} \sqrt{2}\,\mathcal{V}[\rho_t]^{1/2} \notag\\
    &\qquad \leq \frac{C_*}{2}\mathcal{V}[\rho_t]^{1/2}
               + \frac{e^{-\alpha q}}{C_0}\sqrt{2}\,\mathcal{V}[\rho_t]^{1/2} 
               \leq C_* \mathcal{V}[\rho_t]^{1/2},
\end{align*}
which shows \eqref{goal1}.

To verify  \eqref{goal2}, we repeat the same argument with slightly different parameter constraints. Specifically, choose $q, r, \alpha  > 0$ such that
\begin{equation*}
    \frac{(q+\mathcal{E}_r)^\nu}{\eta} \leq 1, \quad
    \frac{e^{-\alpha q}}{C_0}\sqrt{2} \leq 1.
\end{equation*}
Then, we find 
\begin{align*}
    \sup_{x\in B_R(\p)}\|u_\alpha^t(x) - \log_x x^*\|  \leq 1 + 1 \cdot \mathcal{V}[\rho_t]^{1/2}
    \leq \mathcal{V}[\rho_t]^{1/2} + 1,
\end{align*}
as required. 
\end{proof}

\section{Numerical Experiments} \label{sec:5}
\setcounter{equation}{0}

In this section, we present a series of numerical experiments to illustrate the practical performance and theoretical properties of the proposed intrinsic CBO algorithm. We aim to demonstrate the global convergence of the particle system \eqref{eq: particle} on various Riemannian manifolds, including those with non-positive sectional curvature (where our global geometric guarantees apply) and those with positive curvature (where the geometric cutoffs enforce local regularity). 

To simulate the continuous-time interacting stochastic particle system \eqref{eq: particle}, we employ the Riemannian Euler--Maruyama scheme. For a chosen time step $\Delta t > 0$, the position of the $i$-th particle at the $(k+1)$-th iteration, denoted by $x_i^{k+1}$, is updated via the exponential map:
\begin{equation*}
    x_i^{k+1} = \exp_{x_i^k} \left( \Delta t \, \lambda u_\alpha(\rho^{N, k}; x_i^k) + \sqrt{\Delta t} \, \sigma h_{R-\delta, R}(\dist(x_i^k, \p)) \|u_\alpha(\rho^{N, k}; x_i^k)\| V_i^k \right),
\end{equation*}
where $\rho^{N, k}$ is the empirical measure of the particles at step $k$, and $V_i^k \in T_{x_i^k}M$ is a random tangent vector whose components are drawn independently from a standard normal distribution $\mathcal{N}(0,1)$ with respect to a local orthonormal frame.  The complete procedure for our intrinsic CBO method, incorporating the stable Log-Sum-Exp trick and localized geometric cutoffs, is summarized in Algorithm \ref{alg: CBO_manifold}.  In the following, the CBO hyperparameters are set to $\alpha = 10^9$, $\lambda = 1.0$, $\sigma = 0.2$, $T=10$ and the time step is $\Delta t = 0.1$.

\begin{algorithm}[htbp]
\caption{Intrinsic Consensus-Based Optimization on Manifolds (with Stable Gibbs Weights)}
\label{alg: CBO_manifold}
\begin{algorithmic}
\REQUIRE Objective function $\mathcal{E}: M \to \mathbb{R}$, number of particles $N$, time step $\Delta t$, total iterations $K$, consensus parameter $\lambda > 0$, exploration parameter $\sigma > 0$, weight parameter $\alpha \gg 1$.
\REQUIRE Geometric parameters: Pole $\p \in M$, cutoff radii $R, \delta$ satisfying $0 < \delta < R < R+\delta < \min\left(\frac{\inj(M)}{2},\frac{\pi}{2\sqrt{\kappa_+}}\right)$.
\STATE \textbf{Initialization:} Sample initial particles $\{x_i^0\}_{i=1}^N$ uniformly from the geodesic ball $B_R(\p)$.

\FOR{$k = 0$ to $K-1$}
    \STATE \textbf{1. Evaluate Energy:}
    \FOR{$j=1, \dots, N$}
        \STATE Compute $E_k^{(j)} = \mathcal{E}(x_j^k)$.
    \ENDFOR

    \STATE \textbf{2. Compute Stable Gibbs Weights (Log-Sum-Exp):}
    \STATE Find $M_{\max} = \max_j (-\alpha E_k^{(j)})$.
    \STATE Compute the stable partition function: $Z_k = \sum_{j=1}^N \exp(-\alpha E_k^{(j)} - M_{\max})$.
    \FOR{each particle $j=1, \dots, N$}
        \STATE $W_k^{(j)} = \frac{\exp(-\alpha E_k^{(j)} - M_{\max})}{Z_k}$.
    \ENDFOR

    \STATE \textbf{3. Calculate Intrinsic Consensus Vector:}
    \FOR{each particle $i=1, \dots, N$}
        \STATE Calculate the weighted tangent vector using the Riemannian logarithmic map:
        \STATE $\displaystyle \tilde{u}_i^k = \sum_{j=1}^N h_{R, R+\delta}(\dist(x_j^k, \p)) \, W_k^{(j)} \log_{x_i^k}(x_j^k)$.
        \STATE Apply the base point cutoff:
        \STATE $u_i^k = h_{R, R+\delta}(\dist(x_i^k, \p)) \, \tilde{u}_i^k \in T_{x_i^k}M$.
    \ENDFOR

    \STATE \textbf{4. Riemannian Euler--Maruyama Update:}
    \FOR{each particle $i=1, \dots, N$}
        \STATE Sample a standard normal tangent vector $V_i^k \sim \mathcal{N}(0, I_d)$ in $T_{x_i^k}M$.
        \STATE Calculate the localized diffusion coefficient:
        \STATE $D_i^k =  h_{R-\delta, R}(\dist(x_i^k, \p)) \, \|u_i^k\|_{x_i^k}$.
        \STATE Formulate the tangent space update vector:
        \STATE $v_i^k = \lambda u_i^k \Delta t + \sigma D_i^k \sqrt{\Delta t} \, V_i^k \in T_{x_i^k}M$.
        \STATE Update the particle position via the Riemannian exponential map:
        \STATE $x_i^{k+1} = \exp_{x_i^k}(v_i^k)$.
    \ENDFOR
\ENDFOR
\RETURN Approximate global minimizer $\arg\min_{x_i^K} \mathcal{E}(x_i^K)$.
\end{algorithmic}
\end{algorithm}

\subsection{Global Optimization on the Sphere $\mathbb{S}^2$}
For our first numerical experiment, we consider the unit sphere $M = \mathbb{S}^2$ embedded in $\mathbb{R}^3$. This manifold has a constant positive sectional curvature $\kappa_+ = 1$ and an injectivity radius of $\mathrm{inj}(\mathbb{S}^2) = \pi$. 

\textbf{Theoretical vs. Numerical Domain.}
As established in our theoretical analysis, the well-posedness and global convergence guarantees require the geometric cutoffs to satisfy $R + \delta < \min\left(\frac{\inj(M)}{2},\frac{\pi}{2\sqrt{\kappa_+}}\right)$ which ensures the Lipschitz continuity of the logarithmic map. However, in practical numerical optimization, restricting the particle swarm to a local hemisphere is often unnecessary and limits exploration. Therefore, to demonstrate the global robustness of the algorithm, we relax the theoretical cutoffs for this experiment and allow the swarm to explore the entire manifold (effectively setting $h \equiv 1$). To safely handle the cut-locus singularity that occurs when two particles are perfectly antipodal ($d(x,y) = \pi$), we introduce a standard numerical tolerance in the logarithmic map, regularizing the consensus drift to zero if particles cross this exact threshold.

\textbf{Geometric Primitives.}
The intrinsic distance between two points $x, y \in \mathbb{S}^2$ is given by $d(x,y) = \arccos(\langle x, y \rangle)$. The Riemannian logarithmic map and exponential map, which are used to compute the consensus point and update particle positions, are given respectively by:
\begin{align*}
    \log_x(y) &= \frac{d(x,y)}{\sqrt{1 - \langle x, y \rangle^2}} \big(y - \langle x, y \rangle x\big) \in T_x\mathbb{S}^2, \\
    \exp_x(v) &= x \cos(\|v\|) + \frac{v}{\|v\|} \sin(\|v\|) \in \mathbb{S}^2, \quad \text{for } v \in T_x\mathbb{S}^2.
\end{align*}


\textbf{The Objective Function.}
We adapt the classic Ackley function—a highly nonconvex benchmark characterized by a nearly flat outer region and numerous local minima—to the sphere. We define the spherical Ackley function $\mathcal{E}: \mathbb{S}^2 \to \mathbb{R}$ using the ambient coordinates $x = (x_1, x_2, x_3)^\top$:
\begin{equation*}
    \mathcal{E}(x) = -20 \exp\left( -0.2 \sqrt{\frac{x_1^2 + x_2^2}{2}} \right) - \exp\left( \frac{\cos(c \pi x_1) + \cos(c \pi x_2)}{2} \right) + 20 + \exp(1),
\end{equation*}
where $c = 10$ controls the frequency of the local minima and $\exp(1)$ is added to normalize so that its minimum value becomes zero. Note that the spherical Ackley function defined above is perfectly symmetric across the origin. Consequently, it possesses two identical global minima located at the antipodal points $x^o = (0, 0, \pm 1)^\top$, with a minimum value of $\mathcal{E}(x^o) = 0$. To verify that the intrinsic CBO algorithm relies purely on coordinate-free geometric primitives, we construct an offset spherical Ackley function $\mathcal{E}_{x^*}(x) = \mathcal{E}(R_{x^*}^\top x)$. Here, $R_{x^*}$ is a rotation matrix that maps $x^o$ to an arbitrary point $x^* \in \mathbb{S}^2$. This transformation effectively shifts the global minimizer to $x^*$ while perfectly preserving the highly nonconvex, corrugated topography of the energy landscape.

\textbf{Experimental Setup and Results.}
We define the arbitrary global minimizers at $x^* =\pm  \frac{1}{\sqrt{3}}(1, 1, 1)^\top$ and initialize $N = 80$ particles uniformly at random across the entire sphere $\mathbb{S}^2$, representing a state of maximum entropy without geometric confinement.

Despite the presence of numerous local traps and the relaxation of the geometric confinement constraints, the state-dependent diffusion allows the particle swarm to effectively explore the manifold. As the iterations progress, the swarm smoothly collapses toward the target. The algorithm ultimately reaches strict consensus, identifying the global minimizer $x^*$  and empirically validating the robustness of the intrinsic dynamics on unbounded or positively curved manifolds. 

To visualize the spatial performance of the optimization process, Figure \ref{fig:combined_sphere} illustrates the configuration of the particle swarm on the sphere $\mathbb{S}^2$ against the energy landscape of the offset spherical Ackley function $\mathcal{E}_{x^*}$. Initially ($t=0$), the particles are sampled uniformly over the entire sphere $\mathbb{S}^2$ (shown in blue), representing a state of high entropy and high average energy. By the final iteration ($t=T$), the state-dependent diffusion and intrinsic drift have successfully guided the swarm out of the numerous local minima scattered across the landscape. Rather than plotting the entire collapsed swarm, the figure highlights the single algorithmic output, $\arg\min_{x_i^T} \mathcal{E}_{x^*}(x_i^T)$ (shown as a large red circle). This approximate global minimizer is close to one of the true minimizers $ x^*$ (indicated by the green stars). This precise spatial accuracy visually shows the strict convergence properties established in our theoretical analysis and demonstrates the coordinate-free robustness of the intrinsic method.

\begin{figure}[htbp]
    \centering
    \includegraphics[width=0.8\textwidth]{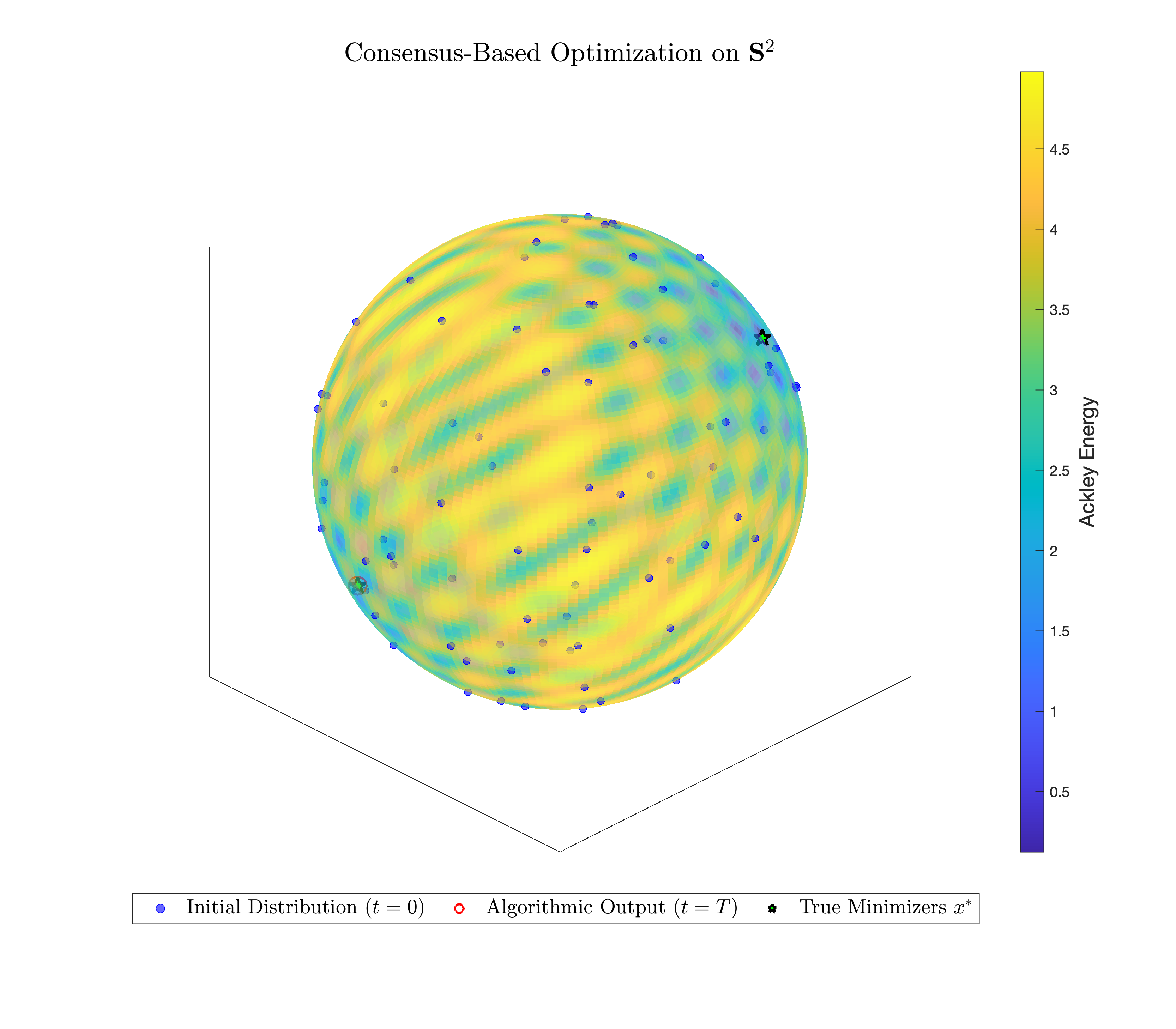} 
    \caption{Spatial configuration of the intrinsic CBO algorithm on $\mathbb{S}^2$. The initial swarm (blue dots) is sampled uniformly over the whole sphere $\mathbb{S}^2$. After $T$ iterations, the algorithm successfully navigates the nonconvex, shifted Ackley landscape. The algorithm returns a single approximate global minimizer (red circle) that precisely aligns with one of the true global minimizers $ x^*$ (green stars), confirming the coordinate-free nature of the geometric dynamics.}
    \label{fig:combined_sphere}
\end{figure}

To quantitatively evaluate the theoretical guarantees of the intrinsic CBO algorithm, Figure \ref{fig:variance} traces the evolution of the empirical variance of the particle swarm. Following the theoretical continuous-time functional $\mathcal{V}[\rho_t]$ introduced in Section \ref{sec:3}, we define the empirical variance for the $N$-particle system as:
\begin{equation*}
    \mathcal{V}^N(t) := \frac{1}{2N} \sum_{i=1}^N \dist(x_i(t), x^*)^2\,.
\end{equation*}

As illustrated in the semi-logarithmic plot, the optimization trajectory exhibits a sharp, sustained exponential decay. This straight-line behavior on the logarithmic scale serves as numerical evidence consistent with Theorem \ref{thm:convergence}, which states that the mean-field variance $\mathcal{V}[\rho_t]$ is bounded by $e^{-Ct} \mathcal{V}[\rho_0]$. The result confirms that the state-dependent diffusion and localized geometric cutoffs successfully enforce the collapse of the particle distribution toward the global minimum, preserving the exponential convergence rate predicted by the theory.

\begin{figure}[htbp]
    \centering
    \includegraphics[width=0.7\textwidth]{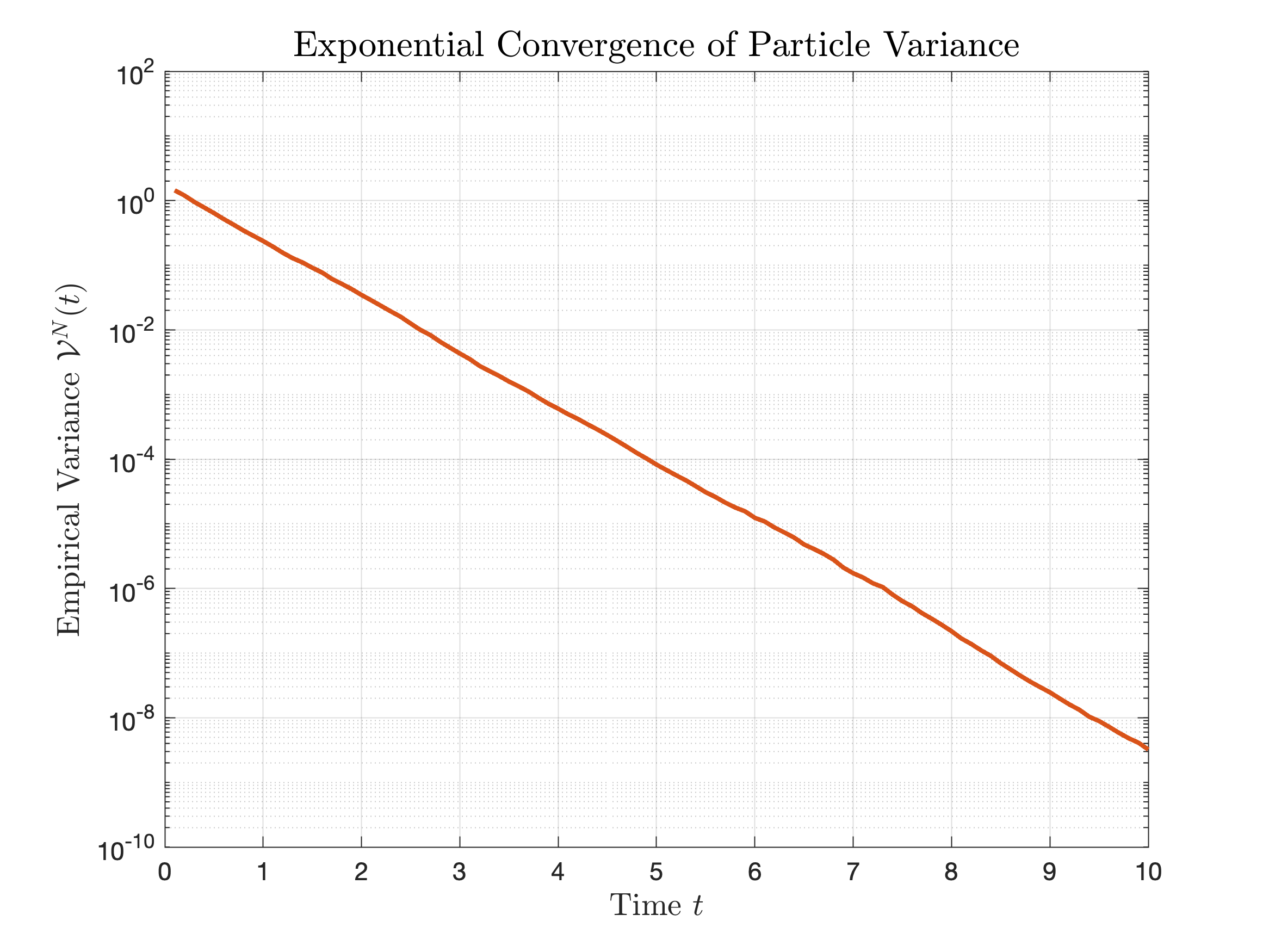} 
    \caption{Empirical variance $\mathcal{V}^N(t)$ of the $N=80$ particle swarm on $\mathbb{S}^2$, plotted on a logarithmic y-axis. The strictly linear downward trend demonstrates the exponential decay of the particle variance, consistent with the global convergence bound established in Theorem \ref{thm:convergence}.}
    \label{fig:variance}
\end{figure}

\subsection{Global Optimization on Hyperbolic Space $\mathbb{H}^2$}
To demonstrate the performance of the intrinsic CBO algorithm on manifolds with negative sectional curvature, we consider the two-dimensional hyperbolic space $\mathbb{H}^2$. This provides a crucial theoretical contrast to the sphere: because $\mathbb{H}^2$ possesses constant negative curvature ($\kappa \equiv -1$) and an infinite injectivity radius, the manifold has no cut locus. Consequently, the logarithmic map is globally Lipschitz continuous. This guarantees that our theoretical mean-field limit and global convergence bounds hold across the \emph{entire} unbounded manifold without the need for the localized geometric cutoffs $h$.

\textbf{Geometric Primitives.}
We employ the standard hyperboloid (or Lorentz) model, embedding $\mathbb{H}^2$ into the Minkowski space $\mathbb{R}^{2+1}$. Points $x \in \mathbb{H}^2$ are defined as $\{ x \in \mathbb{R}^3 \mid \langle x, x \rangle_{\mathcal{L}} = -1, x_3 > 0 \}$, where the Lorentz inner product is given by $\langle x, y \rangle_{\mathcal{L}} := x_1 y_1 + x_2 y_2 - x_3 y_3$.

The intrinsic geodesic distance between any two points is computed as $d(x,y) = \operatorname{arccosh}(-\langle x, y \rangle_{\mathcal{L}})$. The Riemannian logarithmic and exponential maps are globally smooth and given by:
\begin{align*}
    \log_x(y) &= \frac{d(x,y)}{\sinh(d(x,y))} \big(y + \langle x, y \rangle_{\mathcal{L}} x\big) \in T_x\mathbb{H}^2, \\
    \exp_x(v) &= x \cosh(\|v\|_{\mathcal{L}}) + \frac{v}{\|v\|_{\mathcal{L}}} \sinh(\|v\|_{\mathcal{L}}) \in \mathbb{H}^2, \quad \text{for } v \in T_x\mathbb{H}^2,
\end{align*}
where the norm on the tangent space is induced by the Lorentz metric, $\|v\|_{\mathcal{L}} = \sqrt{\langle v, v \rangle_{\mathcal{L}}}$.

\textbf{The Objective Function.}
To test the algorithm's ability to escape local minima in an exponentially expanding volume, we define a highly nonconvex, coordinate-free Radial Ackley function on $\mathbb{H}^2$:
\begin{equation*}
    \mathcal{E}_{x^*}(x) = -20 \exp\left(-0.2 \, d(x, x^*)\right) - \exp\left(\cos(c \pi \, d(x, x^*))\right) + 20 + \exp(1),
\end{equation*}
where $c = 10$ controls the frequency of the local minima. This creates an unbounded funnel containing infinitely many concentric rings of local minima surrounding an arbitrary global minimizer $x^* \in \mathbb{H}^2$.

\textbf{Experimental Setup and Results.}
We define the global minimizer at $x^* = (2, 2, 3)^\top$ and initialize $N = 150$ particles drawn from a highly dispersed Gaussian distribution mapped to the manifold. This ensures the swarm is initialized far from the target in a state of maximum entropy. 

To visualize the spatial dynamics on the negatively curved manifold, Figure \ref{fig:hyperbolic_results} (left)  illustrates the configuration of the particle swarm on the hyperboloid. The initial swarm (blue dots) is widely scattered across the upper domain. Driven by the state-dependent diffusion, the particles successfully navigate the exponentially expanding volume, cascading down the funnel and overcoming the concentric local traps of the Radial Ackley landscape. By the final iteration, the algorithm outputs a single approximate global minimizer (red circle) that is close to the true target $x^*$ (green star).

Figure \ref{fig:hyperbolic_results} (right) quantitatively validates this performance by tracking the empirical variance $\mathcal{V}^N(t) = \frac{1}{2N} \sum_{i=1}^N d_{\mathbb{H}^2}(x_i(t), x^*)^2$ on a semi-logarithmic scale. As predicted by our theory, the variance exhibits a  linear decay on the semilogarithmic scale, establishing the exponential rate of convergence across the unbounded manifold. Note that the variance reaches a strict horizontal floor near $10^{-6}$ at $t \approx 8$. This plateau is a standard numerical artifact representing the discretization error limit of the Euler--Maruyama scheme for the chosen time step $\Delta t$, rather than a failure of the intrinsic consensus dynamics.

\begin{figure}[htbp]
    \centering
    \begin{subfigure}[b]{0.48\textwidth}
        \centering
        \includegraphics[width=\textwidth]{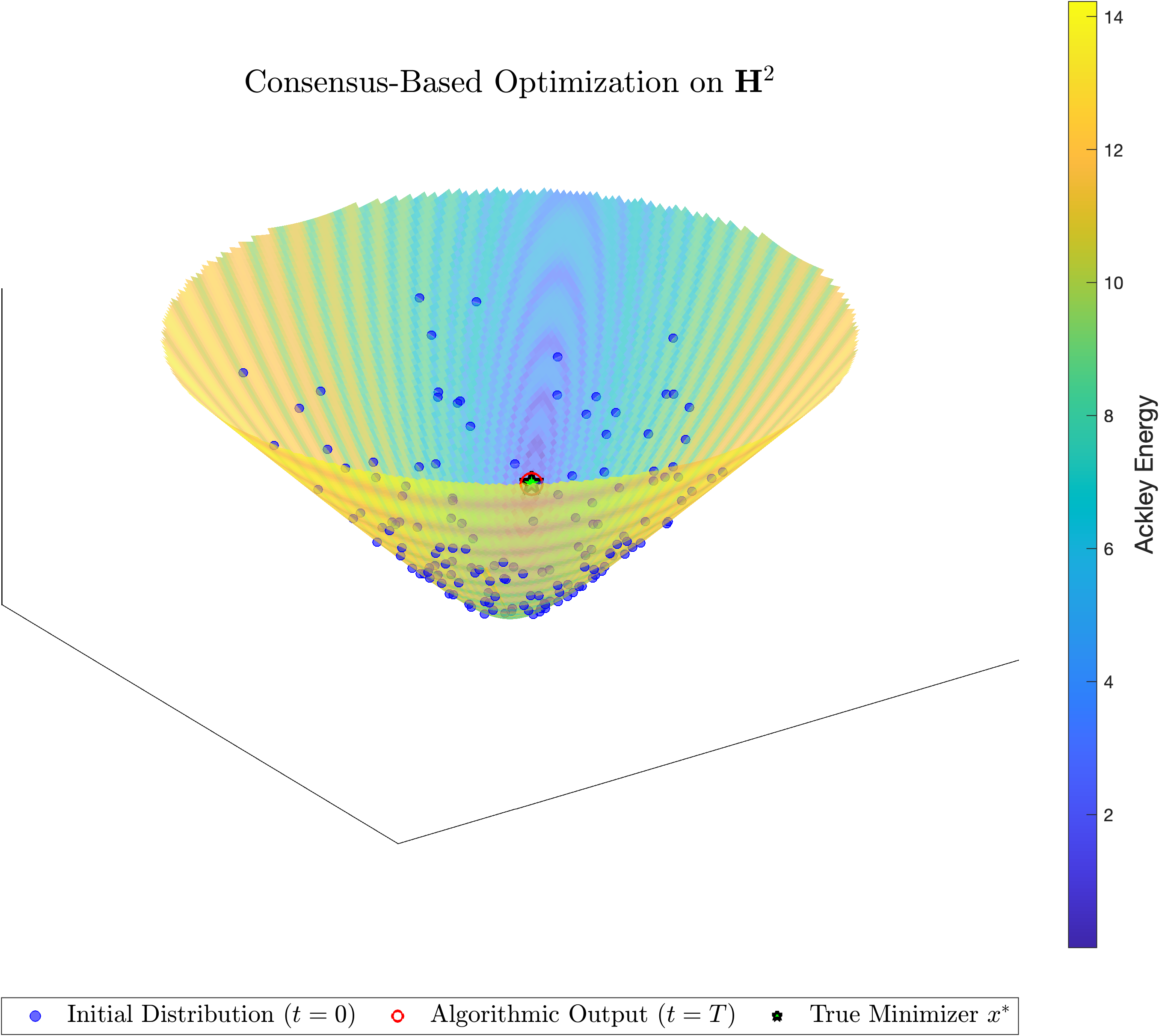}
        \label{fig:hyp_spatial}
    \end{subfigure}
    \hfill
    \begin{subfigure}[b]{0.48\textwidth}
        \centering
        \includegraphics[width=\textwidth]{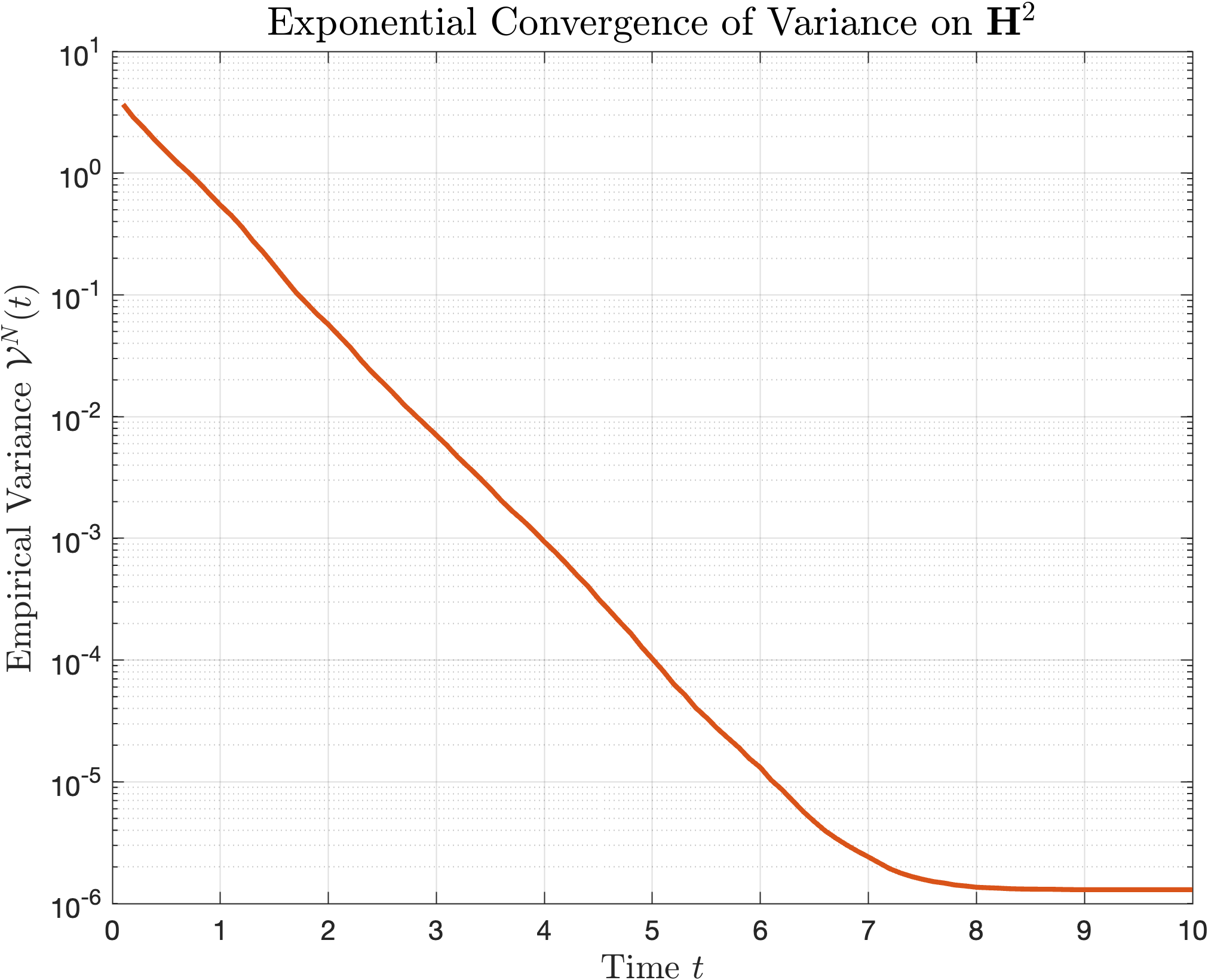}
        \label{fig:hyp_variance}
    \end{subfigure}
    \caption{Numerical results for the intrinsic CBO algorithm on the hyperbolic space $\mathbb{H}^2$. \textbf{(Left)} The algorithm smoothly navigates the globally unbounded Radial Ackley landscape, with the final algorithmic output (red circle) precisely identifying the true global minimizer $x^*$ (green star). \textbf{(Right)} The empirical variance plotted on a semi-logarithmic scale. The strictly linear downward trajectory is consistent with the exponential convergence guarantees established in our mean-field theory, operating here without any localized geometric cutoffs.}
    \label{fig:hyperbolic_results}
\end{figure}

\subsection{Global Optimization on the special orthogonal group $SO(3)$}
To demonstrate the flexibility of the intrinsic CBO algorithm on compact Lie groups, we evaluate its performance on the special orthogonal group $SO(3)$, the manifold of $3 \times 3$ rotation matrices. Optimization on $SO(3)$ is central to numerous applications in robotics, computer vision, and aerospace engineering. Like the sphere, $SO(3)$ is a compact manifold with positive curvature and possesses a cut locus (occurring at rotations of exactly $\pi$ radians). Consequently, our theoretical bounds requiring geometric cutoffs formally apply. However, consistent with our findings on $\mathbb{S}^2$, we relax these cutoffs in practice and allow the swarm to explore the entire group, regularizing the logarithmic map at the exact cut locus to ensure numerical stability.

\textbf{Geometric Primitives.}
The manifold is defined as $SO(3) = \{ R \in \mathbb{R}^{3 \times 3} \mid R^\top R = I, \det(R) = 1 \}$. The tangent space at a rotation $R$ is given by $T_R SO(3) = \{ R \Omega \mid \Omega \in \mathfrak{so}(3) \}$, where $\mathfrak{so}(3)$ is the Lie algebra of $3 \times 3$ skew-symmetric matrices. We equip $SO(3)$ with the standard bi-invariant Riemannian metric, defined via the Frobenius inner product $\langle X, Y \rangle_R = \frac{1}{2} \text{tr}(X^\top Y)$.

The intrinsic geodesic distance between two rotation matrices $R_1, R_2 \in SO(3)$ is the absolute angle of rotation required to align them: $d(R_1, R_2) = \frac{1}{\sqrt{2}} \| \log(R_1^\top R_2) \|_F$, where $\log$ denotes the principal matrix logarithm. The Riemannian logarithmic and exponential maps are elegantly defined via their matrix counterparts:
\begin{align*}
    \log_{R_1}(R_2) &= R_1 \log(R_1^\top R_2) \in T_{R_1}SO(3), \\
    \exp_R(V) &= R \exp(R^\top V) \in SO(3), \quad \text{for } V \in T_R SO(3).
\end{align*}

\textbf{The Objective Function.}
To maintain a consistent benchmark across all manifold topologies, we define a highly nonconvex, coordinate-free Ackley function on $SO(3)$. Given an arbitrary target rotation matrix $R^* \in SO(3)$, the energy of a candidate rotation $R$ is formulated entirely in terms of the geodesic distance $\theta = d(R, R^*) \in [0, \pi]$:
\begin{equation*}
    \mathcal{E}_{R^*}(R) = -20 \exp(-0.2 \theta) - \exp(\cos(c \pi \theta)) + 20 + \exp(1),
\end{equation*}
where $c = 10$. This generates a multimodal energy landscape defined over the group of rotations, featuring deep concentric traps in the rotation angle space.

\textbf{Experimental Setup and Results.}
We define an arbitrary global minimizer $R^*$ and initialize $N = 150$ particles drawn uniformly at random from the Haar measure on $SO(3)$. This unbiased initialization ensures a state of maximum orientational entropy, completely covering the manifold.

\begin{figure}[htbp]
    \centering
    \includegraphics[width=0.95\textwidth]{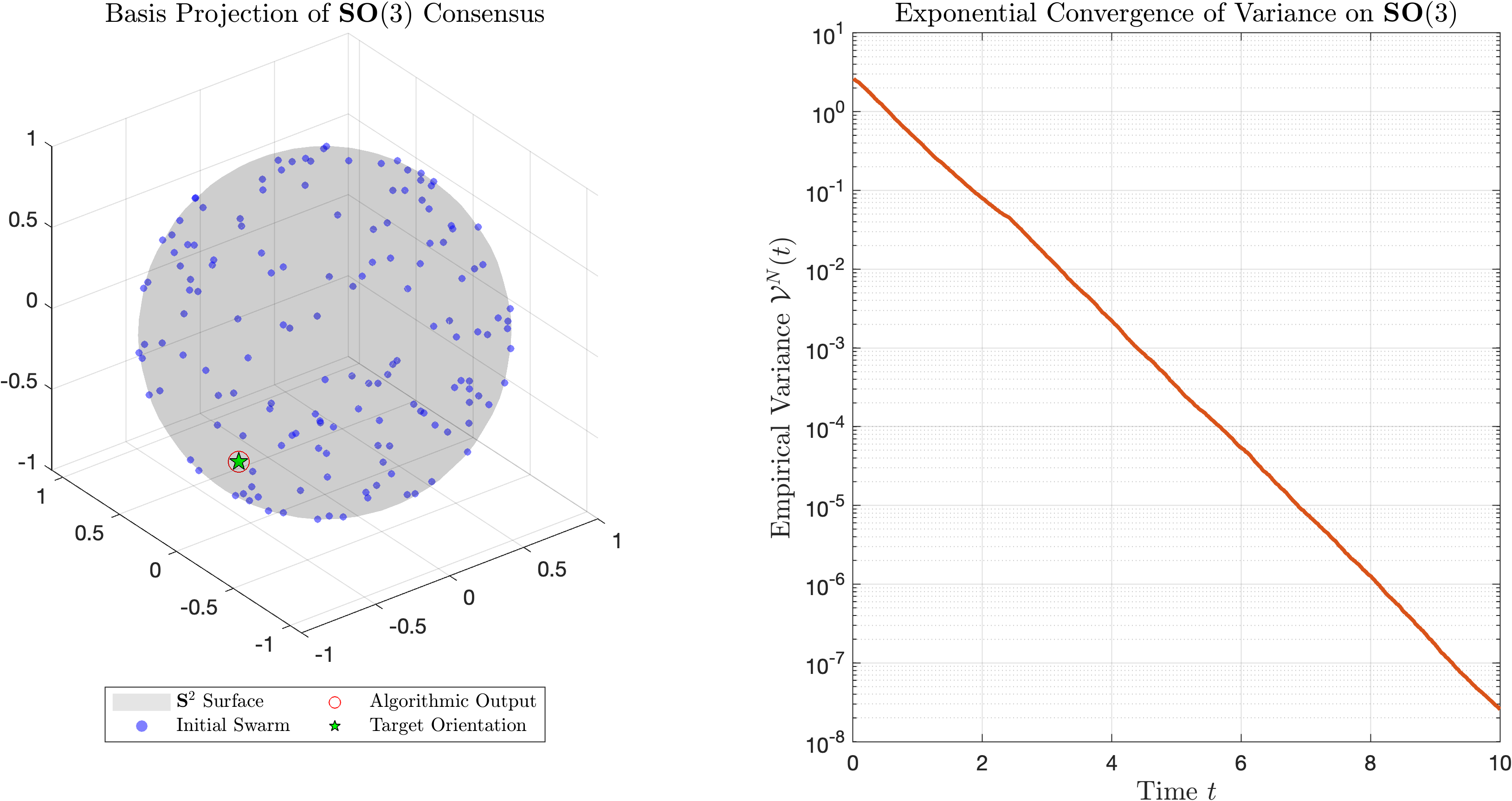} 
    \caption{Numerical results for the intrinsic CBO algorithm on the special orthogonal group $SO(3)$. \textbf{Left:} Spatial projection of the first basis vector of the rotation matrices. The initially uniform swarm (blue dots) reaches strict consensus, with the final algorithmic output (red circle) capturing the target orientation (green star). \textbf{Right:} The empirical variance exhibits strict exponential decay, illustrating our convergence theorem.}
    \label{fig:SO3_results}
\end{figure}

Because a $3 \times 3$ rotation matrix cannot be directly visualized in 3D space, Figure \ref{fig:SO3_results} (left) plots the projection of the first column   vector of each matrix onto the unit sphere $\mathbb{S}^2$. The initial swarm (blue dots) illustrates the uniform Haar distribution prior to optimization. Despite the non-commutative algebraic structure of the Lie group and the highly corrugated energy landscape of the $SO(3)$ Ackley function, the swarm successfully avoids the local minimum traps. The algorithm outputs a final consensus state (red circle) that perfectly   aligns with the target orientation (green star).

Figure \ref{fig:SO3_results} (right) provides quantitative validation by tracking the empirical variance $\mathcal{V}^N(t) = \frac{1}{2N} \sum_{i=1}^N d_{SO(3)}(R_i(t), R^*)^2$. Plotted on a semi-logarithmic scale, the strictly linear downward trajectory empirically confirms the exponential rate of convergence. Notably, this exponential decay is achieved globally without the enforcement of localized geometric cutoffs, confirming that the intrinsic CBO dynamics smoothly generalize to compact matrix Lie groups. 

\newpage
\bibliographystyle{abbrv}
\def\url#1{}
\bibliography{lit.bib,references}

\appendix
\section{Proofs of Lemmas} \label{sec:app.A}
\setcounter{equation}{0}
\subsection{Proof of Lemma \ref{lem:logdiff}}\label{proof:lem:logdiff}
Fix $x\in B_R(\p)$ and set
\[
v_1:=\log_x y_1,\qquad v_2:=\log_x y_2.
\]
By \textbf{(R)}, the ball $B_{R+\delta}(\p)$ is strongly convex. Hence the unique minimizing geodesic $c$ joining $y_1$ and $y_2$ is contained in $B_R(\p)$. In particular, for every $t$,
\[
\dist(x,c(t))\le \dist(x,\p)+\dist(\p,c(t))<2R<\inj(M),
\]
so that $\exp_x^{-1}(c(t))$ is well-defined.

Let $\tilde M$ be the simply connected space form of constant sectional curvature $\kappa_+$,   $\tilde p\in \tilde M$, and let
\[
i:T_xM\to T_{\tilde p}\tilde M
\]
be a linear isometry. Define
\[
\tilde c(t):=\exp_{\tilde p}\circ i\circ \exp_x^{-1}(c(t)).
\]
By applying Lemma \ref{RCT}, we obtain
\[
\dist(y_1,y_2)=L(c)\ge L(\tilde c)
\ge
\dist_{\tilde M}\bigl(\exp_{\tilde p}(i(v_1)),\exp_{\tilde p}(i(v_2))\bigr).
\]
Moreover, since $y_1,y_2\in B_R(\p)$ and $x\in B_R(\p)$, we have
\[
|v_1|=\dist(x,y_1)<2R,
\qquad
|v_2|=\dist(x,y_2)<2R.
\]
In the simply connected space form of constant curvature $\kappa_+$, the differential of the exponential map satisfies the lower bound \cite{Petersen2016}
\[
\dist_{\tilde M}\bigl(\exp_{\tilde p}(i(v_1)),\exp_{\tilde p}(i(v_2))\bigr)
\ge
\frac{\sin(2\sqrt{\kappa_+}R)}{2\sqrt{\kappa_+}R}|v_1-v_2|.
\] 
Therefore, we obtain the desired result:
\[
|\log_x y_1-\log_x y_2|
=
|v_1-v_2|
\le
\frac{2\sqrt{\kappa_+}R}{\sin(2\sqrt{\kappa_+}R)}\,\dist(y_1,y_2).
\]

\subsection{Proof of Lemma \ref{lem:hessian}}\label{proof:lem:hessian}
It suffices to show that
\[
\frac{\d^2}{\d s^2}\left(\frac{1}{2}\dist(x_1(s), x_2(s))^2\right)\bigg|_{s=0}\leq \|v_1-P_{z_1z_2}v_2\|^2+\frac{\kappa \dist(z_1, z_2)^2}{2}\left(\|v_1\|^2+\|v_2\|^2\right).
\]
for $x_1(s)=\exp_{z_1}(sv_1)$ and $x_2(s)=\exp_{z_2}(sv_2)$. From the parametrization, we have $x_1(0)=z_1$ and $x_2(0)=z_2$. Since $\dist(z_1, z_2)<\inj(M)$, we get for $|s|<s_0:=\frac{\inj(M)-\dist(z_1, z_2)}{\|v_1\|+\|v_2\|}$, 
\begin{align*}
\dist\left(x_1(s), x_2(s)\right)&\leq \dist(x_1(s), x_1(0))+\dist(x_1(0), x_2(0))+\dist(x_2(0), x_2(s))\\
&\leq s\|v_1\|+\dist(z_1, z_2)+s\|v_2\|\\
&<\inj(M).
\end{align*}
Hence, the geodesic between the two points $x_1(s)$ and $x_2(s)$ is unique for $|s|<s_0$.

Denote the curve connecting $z_1$ and $z_2$ by $\gamma:[0,1]\to M$ with $\gamma(0)=z_1$ and $\gamma(1)=z_2$. Let $v(t)=P_{\gamma(t)z_1}(tv_1)+P_{\gamma(t)z_2}((1-t)v_2)$ and define a curve $\Gamma_s:[0,1]\to M$ as follows:
\[
\Gamma_s(t) :=\exp_{\gamma(t)}(sv(t)).
\]
For a smooth curve $c$, define the energy of the curve as  
\[
E(c) :=\frac{1}{2}\int_0^1\|\dot{c}(t)\|^2\d t.
\]
In addition, we also define 
\[
F(s):=\Psi(x_1(s), x_2(s))=\frac{1}{2}\dist(x_1(s),x_2(s))^2.
\]
Then, we have
\begin{align}\label{B-1}
0\leq E(\Gamma_s)-F(s),\quad\forall~|s|\leq s_0,
\end{align}
since the geodesic distance is the shortest distance between two points. Since the equality of \eqref{B-1} is achieved at $s=0$, i.e., $F(0)=E(\Gamma_0)$, we know
\[
\frac{\d}{\d s}F(s)\bigg|_{s=0}=\frac{\d}{\d s}E(\Gamma_s)\bigg|_{s=0}=0
\]
and
\[
\frac{\d^2}{\d s^2}F(s)\bigg|_{s=0}\leq\frac{\d^2}{\d s^2}E(\Gamma_s)\bigg|_{s=0}.
\]
From \cite[Theorem 10.22]{lee2018introduction}, we have
\[
\frac{\d^2}{\d s^2}E(\Gamma_s)\bigg|_{s=0}=\int_0^1\left(\|D_tv(t)^\perp\|^2-\langle \mathcal{R}(v(t)^\perp, \gamma'(t))\gamma'(t), v(t)^\perp\rangle\right)\d t,
\]
where $\mathcal{R}$ denotes the curvature tensor and $v(t)^\perp$ is the orthogonal component of $v(t)$ with respect to $\gamma'(t)$, i.e., $v(t)^\perp=v(t)-\langle v(t), T(t)\rangle T(t)$ with $T(t):=\gamma'(t)/\|\gamma'(t)\|$.

Since $D_tv(t)=P_{\gamma(t)x_1}v_1-P_{\gamma(t)x_2}v_2$, we get
\[
\|D_tv(t)^\perp\|^2\leq \|D_tv(t)\|^2\leq \|v_1-P_{x_1x_2}v_2\|^2.
\]
From the curvature condition, we also have
\begin{align*}
-\langle \mathcal{R}(v(t)^\perp, \gamma'(t))\gamma'(t), v(t)^\perp\rangle&\leq\kappa \|\gamma'(t)\|^2\|v(t)^\perp\|^2\\
&\leq \kappa\|\gamma'(t)\|^2\|v(t)\|^2\\
&=\kappa \dist(x_1, x_2)^2\|tv_1+(1-t)P_{x_1x_2}v_2\|^2.
\end{align*}
By combining all the results, we get the desired result:
\begin{align*}
\frac{\d^2}{\d s^2}F(s)\bigg|_{s=0}&\leq\frac{\d^2}{\d s^2}E(\Gamma_s)\bigg|_{s=0}\\
&=\int_0^1\left(\|D_tv(t)^\perp\|^2-\langle \mathcal{R}(v(t)^\perp, \gamma'(t))\gamma'(t), v(t)^\perp\rangle\right)\d t\\
&\leq \|v_1-P_{x_1x_2}v_2\|^2+\kappa \dist(x_1, x_2)^2\int_0^1\|tv_1+(1-t)P_{x_1x_2}v_2\|^2\d t\\
&\leq\|v_1-P_{x_1x_2}v_2\|^2+\kappa \dist(x_1, x_2)^2\int_0^1\left(t\|v_1\|^2+(1-t)\|v_2\|^2\right)\d t\\
&=\|v_1-P_{x_1x_2}v_2\|^2+\frac{\kappa \dist(x_1, x_2)^2}{2}\left(\|v_1\|^2+\|v_2\|^2\right).
\end{align*}

\subsection{Proof of Lemma \ref{lem:transport-diff}}\label{proof:lem:transport-diff}

Using the identity
\[
\frac{D}{dt}\Big(P_{\gamma(t)\to x_1}V(t)\Big)
=
P_{\gamma(t)\to x_1}\Big(\nabla_{\dot\gamma(t)}V(t)\Big),
\]
with $V(t):=\nabla \tilde{\Psi}_y(\gamma(t))$, we obtain
\begin{align*}
\log_{x}y - P_{z}\log_{z}y
&= -\nabla \tilde{\Psi}_y(x) + P_{xz}\nabla \tilde{\Psi}_y(z)\\
&= -\Big(\nabla \tilde{\Psi}_y(x)-P_{xz}\nabla \tilde{\Psi}_y(z)\Big)\\
&= -\int_0^1 P_{\gamma(t)\to x}\Big(\nabla_{\dot\gamma(t)}\nabla \tilde{\Psi}_y(\gamma(t))\Big)\,dt\\
&= -\int_0^1 P_{\gamma(t)\to x}\Big(\mathrm{Hess} \tilde{\Psi}_y(\gamma(t))[\dot\gamma(t)]\Big)\,dt.
\end{align*}
Taking norms and using that parallel transport is an isometry yield
\[
\bigl\|\log_{x}y - P_{xz}\log_{z}y\bigr\|
\le
\int_0^1 \|\mathrm{Hess} \tilde{\Psi}_y(\gamma(t))\|_{\mathrm{op}}\,\|\dot\gamma(t)\|\,dt
\le
\Big(\sup_{x\in B_{R+\delta}(o)}\|\mathrm{Hess}\tilde{\Psi}_y(x)\|_{\mathrm{op}}\Big)\,d(x,z).
\]
It follows from Remark \ref{R2.1} that we get
\[
\sup_{x\in B_{R+\delta}(o)}\|\mathrm{Hess}\tilde{\Psi}_y(x)\|_{\mathrm{op}}\leq 1+\frac{\kappa_-}{2}\times\sup_{x\in B_{R+\delta}(\p)}\dist(x, y)^2\leq1+2\kappa_-(R+\delta)^2=C_1.
\]
This completes the proof.

\subsection{Proof of Lemma \ref{L4.1}}\label{proof:L4.1}
Since $x\in B_{R+\delta}(\p)$, we have
\begin{align*}
\|\tilde{u}_\alpha(\rho;x)\|&=\left\|\frac{1}{\int_{M}w_\alpha(y)\d \rho(y)}\int_{B_{R+\delta}(\p)}h_{R, R+\delta}(\dist(y, \p))w_\alpha(y)\log_xy\d \rho(y)\right\|\\
&\leq\frac{1}{\int_{M}w_\alpha(y)\d \rho(y)}\int_{B_{R+\delta}(\p)}h_{R, R+\delta}(\dist(y, \p))w_\alpha(y)\left\|\log_xy\right\|\d \rho(y)\\
&\leq\frac{1}{\int_{M}w_\alpha(y)\d \rho(y)}\int_{B_{R+\delta}(\p)}w_\alpha(y)\dist(x, y)\d \rho(y).
\end{align*}
Here, we used $h_{R, R+\delta}(\cdot)\leq1$ and $\|\log_xy\|=\dist(x, y)$. Finally, we use
\[
\dist(x, y)\leq \dist(\p, x)+\dist(\p, y)\leq2(R+\delta),\quad\forall~x,y\in B_{R+\delta}(\p)
\]
to conclude the desired boundedness
\[
\|\tilde{u}_\alpha(\rho;x)\|\leq\frac{2(R+\delta)}{\int_{M}w_\alpha(y)\d \rho(y)}\int_{B_{R+\delta}(\p)}w_\alpha(y)\d \rho(y)\leq2(R+\delta).
\]
In the last inequality, we used the non-negativity of $w_\alpha$ and $B_{R+\delta}(\p)\subset M$.

\subsection{Proof of Lemma \ref{L4.2}}\label{proof:L4.2}
We observe 
\begin{align*}
&\|u_\alpha(\rho;x_1)-P_{x_1x_2}u_\alpha(\rho;x_2)\| = \|h_{R,R+\delta}(\dist(x_1, \p))\tilde{u}_\alpha(\rho;x_1)-h_{R,R+\delta}(\dist(x_2, \p))P_{x_1x_2}\tilde{u}_\alpha(\rho;x_2)\|\\
&\leq\|(h_{R, R+\delta}(\dist(x_1, \p))-h_{R, R+\delta}(\dist(x_2, \p)))\tilde{u}_\alpha(\rho;x_1)\| \\
&\hspace{0.5cm}+h_{R, R+\delta}(\dist(x_2, \p))\|\tilde{u}_\alpha(\rho;x_1)-P_{x_1x_2}\tilde{u}_\alpha(\rho;x_2)\|\\
& = |h_{R, R+\delta}(\dist(x_1, \p))-h_{R, R+\delta}(\dist(x_2, \p))|\cdot \|\tilde{u}_\alpha(\rho;x_1)\| \\
&\hspace{0.5cm}+h_{R, R+\delta}(\dist(x_2, \p))\|\tilde{u}_\alpha(\rho;x_1)-P_{x_1x_2}\tilde{u}_\alpha(\rho;x_2)\| \\
&=: \mathcal I_{61} + \mathcal I_{62}.
\end{align*}
$\bullet$ (Estimate of $\mathcal I_{61}$): We use Lemma \ref{L4.1} to find 
\begin{align*}
\mathcal I_{61} & = |h_{R, R+\delta}(\dist(x_1, \p))-h_{R, R+\delta}(\dist(x_2, \p))|\cdot \|\tilde{u}_\alpha(\rho;x_1)\| \\
&\leq \mathrm{Lip}(h_{R, R+\delta})|\dist(x_1, \p)-\dist(x_2, \p)|\cdot (2(R+\delta))\\
&\leq2(R+\delta)\mathrm{Lip}(h_{R, R+\delta})\dist(x_1, x_2).
\end{align*}

\noindent $\bullet$ (Estimate of $\mathcal I_{62}$): By recalling the definition of the parallel transport, we see
\begin{align*}
&\left\|\tilde{u}_\alpha(\rho;x_1)-P_{x_1 x_2}\tilde{u}_\alpha(\rho;x_2)\right\|\\
&\hspace{0.5cm}=\frac{1}{\int_M w_\alpha(y)\d\rho(y)}\left\|\int_{B_{R+\delta}(\p)}h_{R, R+\delta}(\dist(y, \p))w_\alpha(y)\bigg(\log_{x_1}y-P_{x_1x_2}\log_{x_2}y\bigg)\d \rho(y)\right\|\\
&\hspace{0.5cm}\leq\frac{1}{\int_M w_\alpha(y)\d\rho(y)}\int_{B_{R+\delta}(\p)}h_{R, R+\delta}(\dist(y, \p))w_\alpha(y)\bigg\|\log_{x_1}y-P_{x_1x_2}\log_{x_2}y\bigg\|\d \rho(y)\\
&\hspace{0.5cm}\leq\frac{C_1\dist(x_1, x_2)}{\int_M w_\alpha(y)\d\rho(y)}\int_{B_{R+\delta}(\p)}h_{R, R+\delta}(\dist(y, \p))w_\alpha(y)\d \rho(y)\\
&\hspace{0.5cm}\leq C_1\dist(x_1, x_2)
\end{align*}
where we use Lemma \ref{lem:transport-diff} for the fourth line and the fact that $0\leq h_{R,R+\delta}(\cdot)\leq 1$ for the fifth line. 

To this end, we  collect the above estimates to get
\[
\|u_\alpha(\rho;x_1)-P_{x_1x_2}u_\alpha(\rho;x_2)\|\leq \left(C_1+2(R+\delta)\mathrm{Lip}(h_{R,R+\delta})\right)\dist(x_1,x_2).
\]
Since $C_1+2\mathrm{Lip}(h_{R,R+\delta})(R+\delta)$ is the largest coefficient, we get the desired result with
\[
C_2=C_1+2\mathrm{Lip}(h_{R,R+\delta})(R+\delta).
\]

\subsection{Proof of Lemma \ref{L4.3}}\label{proof:L4.3}
If $x\in B_{R+\delta}(\p)^c$ then
\[
\|{u}_\alpha(\rho_1;x)-{u}_\alpha(\rho_2;x)\|=0.
\]
Thus, this case is done.  Now, we assume $x\in B_{R+\delta}(\p)$. Then, we have
\[
\|{u}_\alpha(\rho_1;x)-{u}_\alpha(\rho_2;x)\|=h_{R,R+\delta}(\dist(x, \p))\|\tilde{u}_\alpha(\rho_1;x)-\tilde{u}_\alpha(\rho_2;x)\|\leq\|\tilde{u}_\alpha(\rho_1;x)-\tilde{u}_\alpha(\rho_2;x)\|.
\]
The difference $\tilde{u}_\alpha(\rho_1;x)-\tilde{u}_\alpha(\rho_2;x)$ can be simplified into
\begin{align*}
&\tilde{u}_\alpha(\rho_1;x)-\tilde{u}_\alpha(\rho_2;x)\\
=&\frac{1}{\int_{M}w_\alpha(y)\d \rho_1(y)}\int_{B_{R+\delta}(\p)}h_{R, R+\delta}(\dist(y, \p))w_\alpha(y)\log_xy\d \rho_1(y)\\
&\hspace{3cm}-\frac{1}{\int_{M}w_\alpha(y)\d \rho_2(y)}\int_{B_{R+\delta}(\p)}h_{R, R+\delta}(\dist(y, \p))w_\alpha(y)\log_xy\d \rho_2(y)
\\
=&\frac{1}{\int_{M}w_\alpha(y)\d \rho_1(y)}\bigg(\int_{B_{R+\delta}(\p)}h_{R, R+\delta}(\dist(y, \p))w_\alpha(y)\log_xy\d \rho_1(y)\\
&\hspace{5cm}-\int_{B_{R+\delta}(\p)}h_{R, R+\delta}(\dist(y, \p))w_\alpha(y)\log_xy\d \rho_2(y)\bigg)\\
&+\left(\frac{1}{\int_{M}w_\alpha(y)\d \rho_1(y)}-\frac{1}{\int_{M}w_\alpha(y)\d \rho_2(y)}\right)\int_{B_{R+\delta}(\p)}h_{R, R+\delta}(\dist(y, \p))w_\alpha(y)\log_xy\d \rho_2(y)\\
=&\frac{1}{\int_{M}w_\alpha(y)\d \rho_1(y)}\int_{B_{R+\delta}(\p)}h_{R, R+\delta}(\dist(y, \p))w_\alpha(y)\log_xy(\rho_1(y)-\rho_2(y))\d y\\
&+\frac{\int_{M}w_\alpha(y)(\rho_2(y)-\rho_1(y))\d y}{\int_{M}w_\alpha(y)\d \rho_1(y)\int_{M}w_\alpha(y)\d \rho_2(y)}\int_{B_{R+\delta}(\p)}h_{R, R+\delta}(\dist(y, \p))w_\alpha(y)\log_xy\d \rho_2(y)\\
=&:\mathcal{I}_{71}+\mathcal{I}_{72}.
\end{align*}

We estimate $\mathcal{I}_{72}$ and $\mathcal{I}_{71}$ one by one.\\

\noindent $\bullet$ (Estimate of $\mathcal I_{72}$): We observe
\begin{align*}
\|\mathcal{I}_{72}\|&\leq\left\|\frac{\int_{M}w_\alpha(y)(\rho_2(y)-\rho_1(y))\d y}{\int_{M}w_\alpha(y)\d \rho_1(y)\int_{M}w_\alpha(y)\d \rho_2(y)}\int_{B_{R+\delta}(\p)}h_{R, R+\delta}(\dist(y, \p))w_\alpha(y)\log_xy\d \rho_2(y)\right\|\\
&\leq\frac{\left|\int_{M}w_\alpha(y)(\rho_2(y)-\rho_1(y))\d y\right|}{\int_{M}w_\alpha(y)\d \rho_1(y)\int_{M}w_\alpha(y)\d \rho_2(y)}\int_{B_{R+\delta}(\p)}h_{R, R+\delta}(\dist(y, \p))w_\alpha(y)\dist(x, y)\d \rho_2(y)\\
&\leq\frac{2(R+\delta)}{\int_{M}w_\alpha(y)\d \rho_1(y)}\left|\int_{M}w_\alpha(y)(\rho_2(y)-\rho_1(y))\d y\right|.
\end{align*}
In the second and third inequalities, we used $\|\log_xy\|=\dist(x, y)\leq 2(R+\delta)$. To estimate the remaining integration, recall the Kantorovich--Rubinstein duality \cite{Villani}:
\[
W_1(\rho_1,\rho_2)=\sup_{\phi\in \mathrm{Lip}_1(M)}\left(\int_M\phi(y)(\rho_1(y)-\rho_2(y))\d y\right),
\]
where $\mathrm{Lip}_1(M)$ is the set of 1-Lipschitz functions on $M$, i.e. $\phi\in \mathrm{Lip}_1(M)$ if and only if
\[
|\phi(x_1)-\phi(x_2)|\leq \dist(x_1, x_2),\quad\forall~x_1, x_2\in M.
\]
Hence, we find
\[
\left|\int_Mw_\alpha(y)(\rho_2(y)-\rho_1(y))\d y\right|\leq \mathrm{Lip}(w_\alpha)W_1(\rho_1, \rho_2).
\]
 Here, $w_\alpha$ is Lipschitz since the objective function $\mathcal{E}$ is Lipschitz. Finally, we obtain 
\begin{align*}
\|\mathcal{I}_{72}\|&\leq \frac{2(R+\delta)}{\int_M w_\alpha(y)\d\rho_1(y)}\mathrm{Lip}(w_\alpha)W_1(\rho_1, \rho_2)
\leq \left(\frac{2(R+\delta)\mathrm{Lip}(w_\alpha)}{\inf w_\alpha}\right)W_1(\rho_1, \rho_2).
\end{align*}\\

\noindent $\bullet$ (Estimate of $\mathcal I_{71}$):  To estimate $\mathcal{I}_{71}$, we consider a transport plan $\pi$ whose marginals are $\rho_1$ and $\rho_2$. Then, we have
\begin{align*}
&\int_{B_{R+\delta}(\p)}h_{R, R+\delta}(\dist(y, \p))w_\alpha(y)\log_xy(\rho_1(y)-\rho_2(y))\d y\\
=&\int_{B_{R+\delta}(\p)\times B_{R+\delta}(\p)}\left(h_{R, R+\delta}(\dist(y_1, \p))w_\alpha(y_1)\log_xy_1-h_{R, R+\delta}(\dist(y_2, \p))w_\alpha(y_2)\log_xy_2\right)\d\pi(y_1, y_2)
\end{align*}
The integrand can be decomposed into
\begin{align*}
&h_{R, R+\delta}(\dist(y_1, \p))w_\alpha(y_1)\log_xy_1-h_{R, R+\delta}(\dist(y_2, \p))w_\alpha(y_2)\log_xy_2\\
&\quad=\bigg(h_{R, R+\delta}(\dist(y_1, \p))-h_{R, R+\delta}(\dist(y_2, \p))\bigg)w_\alpha(y_1)\log_xy_1\\
&\quad\quad+h_{R, R+\delta}(\dist(y_2, \p))\bigg(w_\alpha(y_1)-w_\alpha(y_2)\bigg)\log_xy_1\\
&\quad \quad +h_{R,R+\delta}(\dist(y_2,\p))w_\alpha(y_2)\bigg(\log_xy_1-\log_xy_2\bigg).
\end{align*}
Then, we   estimate the whole integrand as
\begin{align*}
&\|h_{R, R+\delta}(\dist(y_1, \p))w_\alpha(y_1)\log_xy_1-h_{R, R+\delta}(\dist(y_2, \p))w_\alpha(y_2)\log_xy_2\|\\
&\quad \leq  \mathrm{Lip}(h_{R,R+\delta})\dist(y_1, y_2)(\sup w_\alpha)\dist(x, y_1)+\mathrm{Lip}(w_\alpha)\dist(y_1, y_2)\dist(x, y_1)+(\sup w_\alpha)\dist(y_1, y_2)\\
&\quad = \bigg(\mathrm{Lip}(h_{R,R+\delta})(\sup w_\alpha)\dist(x, y_1)+\mathrm{Lip}(w_\alpha)\dist(x, y_1)+(\sup w_\alpha)\bigg)\dist(y_1, y_2)\\
&\quad \leq\bigg(\sup w_\alpha\big(2(R+\delta)\mathrm{Lip}(h_{R,R+\delta})+1\big)+2(R+\delta)\mathrm{Lip}(w_\alpha)\bigg)\dist(y_1, y_2).
\end{align*}
Hence, we can estimate $\mathcal{I}_{71}$ as follows:
\begin{align*}
\|\mathcal{I}_{71}\|&\leq\frac{1}{\int_{M}w_\alpha(y)\d \rho_1(y)}\left\|\int_{B_{R+\delta}(\p)}h_{R, R+\delta}(\dist(y, \p))w_\alpha(y)\log_xy(\rho_1(y)-\rho_2(y))\d y\right\|\\
&\leq\frac{1}{\inf w_\alpha}\bigg(\sup w_\alpha\big(2(R+\delta)\mathrm{Lip}(h_{R,R+\delta})+1\big)+2(R+\delta)\mathrm{Lip}(w_\alpha)\bigg)\\
&\hspace{6cm}\times\int_{B_{R+\delta}(\p)\times B_{R+\delta}(\p)}\dist(y_1, y_2)\d\pi(y_1, y_2).
\end{align*}
Since the above inequality holds for an arbitrary transport plan $\pi$, we get
\[
\|\mathcal{I}_{71}\|\leq\frac{1}{\inf w_\alpha}\bigg(\sup w_\alpha\big(2(R+\delta)\mathrm{Lip}(h_{R,R+\delta})+1\big)+2(R+\delta)\mathrm{Lip}(w_\alpha)\bigg) W_1(\rho_1, \rho_2).
\]
Finally, we obtain the desired result:
\begin{align*}
&\|\tilde{u}_\alpha(\rho_1;x)-\tilde{u}_\alpha(\rho_2;x)\|\\
&\quad \leq  \|\mathcal{I}_{71}\|+\|\mathcal{I}_{72}\|\\
&\quad \leq\frac{1}{\inf w_\alpha}\bigg(\sup w_\alpha\big(2(R+\delta)\mathrm{Lip}(h_{R,R+\delta})+1\big)+4(R+\delta)\mathrm{Lip}(w_\alpha)\bigg)W_1(\rho_1, \rho_2).
\end{align*}

\subsection{Proof of Lemma \ref{lem: Lip-particle}}\label{proof:lem: Lip-particle} We observe 
\begin{align*}
&\|u_\alpha(\rho^{N, X};x_i)-P_{x_iy_i}u_\alpha(\rho^{N, Y};y_i)\|\\
&\quad \leq\|u_\alpha(\rho^{N, X};x_i)-P_{x_iy_i}u_\alpha(\rho^{N, Y};x_i)\|+\|u_\alpha(\rho^{N, Y};x_i)-P_{x_iy_i}u_\alpha(\rho^{N, Y};y_i)\|\\
&\quad \leq C_3 W_1(\rho^{N,X}, \rho^{N, Y})+C_2\dist(x_i, y_i).
\end{align*}
Since the following relation holds
\[
W_1(\rho^{N,X},\rho^{N, Y})\leq \frac{1}{N}\sum_{i=1}^N\dist(x_i, y_i)\leq D(X, Y),
\]
this completes the proof:
\[
\|u_\alpha(\rho^{N, X};x_i)-P_{x_iy_i}u_\alpha(\rho^{N, Y};y_i)\|\leq (C_2+C_3)D(X, Y).
\]

\subsection{Proof of $\nabla_M\phi_\epsilon(\dist(x, \p))\cdot u_\alpha^t(x)\geq0$}\label{secA:8}

By our choice of $\phi_\epsilon$, we have
\[
\nabla_M\phi_\epsilon(\dist(x,\p))=0
\]
for all $x$ satisfying either $\dist(x,\p)<R+\delta$ or $\dist(x,\p)>R+\delta+\epsilon$. Thus, we assume
\[
R+\delta\leq \dist(x, \p)\leq R+\delta+\epsilon<\inj(M).
\]
Then, we have
\[
\nabla_M\phi_\epsilon(\dist(x, \p))=-\left(\frac{\phi_\epsilon'(\dist(x, \p))}{\dist(x, \p)}\right)\log_xo
\]
with $\phi_\epsilon'(\dist(x, \p))\leq0$. Since  the inner product between $\nabla_M\phi_\epsilon(\dist(x, \p))$ and $u_\alpha^t(x)$ can be calculated as
\begin{align}\label{dotproduct-estimate}
\begin{aligned}
&\nabla_M\phi_\epsilon(\dist(x, \p))\cdot u_\alpha^t(x)\\
&=-\left(\frac{\phi_\epsilon'(\dist(x, \p))}{\dist(x, \p)}\right)\log_x\p\cdot u_\alpha^t(x)\\
&=-\left(\frac{\phi_\epsilon'(\dist(x, \p))}{\dist(x, \p)}\right)\frac{h_{R,R+\delta}(\dist(x,\p))}{\int_M w_\alpha(y)\d\rho(y)}
\int_{B_{R+\delta}(\p)} h_{R,R+\delta}(\dist(y,\p))\,w_\alpha(y)\big(\log_x\p\cdot\log_x y\big)\d\rho(y).
\end{aligned}
\end{align}
The Rauch comparison theorem in Lemma \ref{RCT}  and the spherical cosine law give 
\begin{align*} 
\cos\left(\sqrt{\kappa_+}\dist(y, \p)\right) &\leq\cos\left(\sqrt{\kappa_+}\dist(x, \p)\right)\cos\left(\sqrt{\kappa_+}\dist(x,y)\right) \\
&\quad +\sin\left(\sqrt{\kappa_+}\dist(x, \p)\right)\sin\left(\sqrt{\kappa_+}\dist(x,y)\right)\cos(\angle yx\p).
\end{align*} 
By using 
\begin{align}\label{xy-range}
\dist(y, \p)\leq R+\delta\leq\dist(x, \p)<\frac{\pi}{2\sqrt{\kappa_+}},
\end{align}
we have
\begin{align*} 
\cos\left(\sqrt{\kappa_+}(R+\delta)\right) &\leq \cos\left(\sqrt{\kappa_+}(R+\delta)\right)\cos\left(\sqrt{\kappa_+}\dist(x, y)\right) \\
&\quad +\sin\left(\sqrt{\kappa_+}\dist(x, \p)\right)\sin\left(\sqrt{\kappa_+}\dist(x,y)\right)\cos(\angle yx\p).
\end{align*} 
Since $\cos(\sqrt{\kappa_+}\dist(x, y))\leq 1$, we get $\cos(\angle(yx\p))\geq0$. Therefore, we get $\log_x\p\cdot\log_xy\geq0$ for any $x$ and $y$ satisfying \eqref{xy-range}. Finally, we substitute it into \eqref{dotproduct-estimate} to obtain the desired result:
\[
\nabla_M\phi_\epsilon(\dist(x, \p))\cdot u_\alpha^t(x)\geq0.
\]

\end{document}